\documentclass[11pt,a4paper]{amsart}
\usepackage{amscd,amssymb,amsmath,latexsym}

\usepackage{amsfonts}
\usepackage{graphicx}%
\usepackage[active]{srcltx}
\newtheorem{definition}{Definition}[section]
\newtheorem{notation}[definition]{Notation}
\newtheorem{rema}[definition]{Remark}
\newtheorem{exa}[definition]{Example}
\newtheorem{obs}[definition]{Observation}

\newtheorem{proposition}[definition]{Proposition}
\newtheorem{theorem}[definition]{Theorem}
\newtheorem{corollary}[definition]{Corollary}

\def\ov#1{{\overline{#1}}}

{\begin{rema}\rm}%
{\end{rema} }

\newenvironment{example}%
{\begin{exa}\rm}%
{\end{exa} }

{\begin{obs}\rm}%
{\end{obs} }

\newenvironment{ac}{\noindent{\bf Acknowledgements:}}{}

\newcommand{\Res}{{\operatorname{Res}}}
\newcommand{\Sres}{{\operatorname{Sres}}}

\newcommand{\Id}{{\operatorname{Id}}}

\newcommand{\bfe}{{\boldsymbol{e}}}

\newcommand{\bfx}{{\boldsymbol{x}}}
\newcommand{\bfy}{{\boldsymbol{y}}}

\newcommand{\bfalpha}{{\boldsymbol{\alpha}}}
\newcommand{\bfgamma}{{\boldsymbol{\gamma}}}
\newcommand{\bfbeta}{{\boldsymbol{\beta}}}
\newcommand{\bfxi}{{\boldsymbol{\xi}}}

\newcommand{\bfLambda}{{\boldsymbol{\Lambda}}}

\def\fh{{f}^h}

\def\RR{{\mathcal R}}
\def\SS{{\mathcal S}}
\def\TT{{\mathcal T}}
\def\HH{{\mathcal H}}
\def\EE{{\mathcal E}}

\def\OO{{\mathcal O}}

\def\C{{\mathbb C}}

\def\N{{\mathbb N}}

\def\Z{{\mathbb Z}}

\begin{document}

\title{Subresultants  in Multiple Roots}

\author{Carlos D' Andrea}
\address{Universitat de Barcelona, Departament d'{\`A}lgebra i Geometria.
Gran Via 585, 08007 Barcelona, Spain.}
\email{cdandrea@ub.edu}
\urladdr{http://atlas.mat.ub.es/personals/dandrea}

\author{Teresa Krick}
\address{Departamento de Matem\'atica, Facultad de
Ciencias Exactas y Naturales, Universidad de Buenos Aires and IMAS,
CONICET, Argentina} \email{krick@dm.uba.ar}
\urladdr{http://mate.dm.uba.ar/\~\,krick}

\author{Agnes Szanto}
\address{Department of Mathematics, North Carolina State
University, Raleigh, NC 27695 USA}
\email{aszanto@ncsu.edu}
\urladdr{www4.ncsu.edu/\~\,aszanto }

\begin{abstract}
We extend our previous work on Poisson-like formulas for subresultants in roots  to the case of polynomials with multiple roots in
both the univariate and multivariate case, and also explore some
closed formulas in roots for  univariate polynomials in this
multiple roots setting.
\end{abstract}

\date{\today}
\thanks{Carlos D'Andrea is partially supported
by  the Research Project MTM2007--67493, Teresa Krick was partially
suported by ANPCyT PICT 33671/05, CONICET PIP 2010-2012 and UBACyT
grants, and Agnes Szanto was partially supported by NSF grants
CCR-0347506 and CCF-1217557.}\maketitle

\section{Introduction}
In \cite{DKS2006}  we presented Poisson-like formulas for
multivariate subresultants in terms of the roots of the system given
by all but one of the input polynomials, provided that all the roots
were {\em simple}, i.e. that the ideal generated by these polynomials
 is zero-dimensional and radical. Multivariate resultants
were mainly introduced by Macaulay in \cite{Mac2}, after earlier
work by Euler, Sylvester and Cayley, while multivariate
subresultants were first defined by Gonzalez-Vega in
\cite{Gon2,Gon}, generalizing Habicht's method \cite{Hab}. The
notion of subresultants that we use in this text was introduced by Chardin in
\cite{Cha}.

\smallskip Later on, in \cite{DHKS2007,DHKS2009}, we focused on the
classical univariate case and reworked the relation between
subresultants and  double Sylvester sums, always in the simple roots
case (where double sums are actually well-defined). This is also the
subject of the more recent articles \cite{RS11,KS12}. As one of the
referees of the MEGA'2007 conference pointed out to us, working out
these results for the case of polynomials with multiple roots would
also be interesting.

\smallskip This paper is a first attempt in that direction. We succeed
in describing Poisson like formulas for  univariate and multivariate
subresultants in the presence of multiple roots, as well as to
obtain formulas in roots in the univariate setting for subresultants
 of degree 1 and of degree immediately below the minimum of the degrees of the input polynomials:
 the two non-trivial extremal cases in the  sequence of subresultants.
We cannot generalize these formulas for other intermediate degrees,
and  it is still not clear for us which is the correct way of
generalizing Sylvester double sums in the multiple roots case.

\smallskip
The paper is organized as follows: In Section \ref{2} we recall the
definitions of the classical univariate subresultants and Sylvester
double sums, and of the generalized Wronskian and Vandermonde
matrices. We then show how the Poisson formulas obtained in
\cite{Hon1999} for the subresultants in the case of simple roots
extend to the multiple roots setting by means of these generalized
matrices. We also obtain formulas in roots for subresultants in the
two extremal non-trivial cases mentioned above. In Section \ref{3}
we present Poisson-like formulas for multivariate subresultants in
the case of multiple roots, generalizing our previous results
described in \cite{DKS2006}.

\smallskip

\begin{ac} We wish to thank the referee for her/his careful reading
and comments.  A preliminary version of these results was presented
at the MEGA 2009 Conference in Barcelona. Part of this work was done
at the Fields Institute in Toronto while the authors were
participating in the Fall 2009 Thematic Program on Foundations of
Computational Mathematics.
\end{ac}

\section{Univariate Case: Subresultants in multiple roots}\label{2}

\subsection{Notation}
We first establish a  notation that will make the
presentation of the problem and the state of the art simpler.

\smallskip
Set $d,e\in \N$ and  let $A:=\big(\alpha_1,\ldots,\alpha_d\big)$ and
$ B:=\big(\beta_1,\ldots, \beta_e\big)$  be two (ordered) sets of
$d$ and $e$ different indeterminates respectively.

\smallskip For $m,n\in \N$, set
$(d_1,\ldots,d_m)\in \N^m$ and $ (e_1,\ldots,e_n)\in \N^n$ such that
$d_1+\cdots +d_m=d $ and $  e_1+\cdots +e_n=e$, and let
$$\overline{A}:=\big((\alpha_1,d_1);\ldots; (\alpha_{m},d_m)\big)\quad
\mbox{and} \quad
\overline{B}:=\big((\beta_1,e_1);\ldots;(\beta_{n},e_n)\big)$$
(these will be regarded as ``limit sets'' of $A$ and $B$ when roots
are packed following the corresponding multiplicity patterns).

\smallskip
We associate to  $A$  and $B$  the
 monic polynomials $f$ and $g$  of degrees $d$ and $e$ respectively,
 and the set $R(A,B)$, where
\begin{align*}&f(x):=\prod_{i=1}^d(x-\alpha_i) \quad \mbox{and} \quad
g(x):=\prod_{j=1}^e(x-\beta_j),\\
&R(A,B)=\prod_{1\le i\le d,1\le j\le e}(\alpha_i-\beta_j),
=\prod_{1\le i\le d} g(\alpha_i)\end{align*} with natural limits
when the roots are packed \begin{align*}&
 \overline{f}(x):=\prod_{i=1}^{m}(x-\alpha_i)^{d_i} \quad \mbox{and}
 \quad
\ov g(x):=\prod_{j=1}^n(x-\beta_j)^{e_j} ,\\ &R(\ov A,\ov B)=
 \prod_{1\le i\le m,1\le j\le n}(\alpha_i-\beta_j)^{d_ie_j}=\prod_{1\le i\le m} \ov g(\alpha_i)^{d_i}.\end{align*}

 \subsection{Subresultants and Sylvester double sums}
We recall that for $0\le t\le d< e$ or $0\le t < d=e$, the $t$-th
subresultant of the polynomials $f=a_dx^d+\cdots + a_0$ and
$g=b_ex^e+\cdots +b_0$,  introduced by J.J.~Sylvester in
\cite{{Sylvester:1853}},  is defined as

\begin{equation*}\label{defsub}
\Sres _t(f,g)
:=\det%
\begin{array}{|cccccc|c}
\multicolumn{6}{c}{\scriptstyle{d+e-2t}}\\
\cline{1-6}
a_{d} & \cdots & & \cdots & a_{t+1-\left(e-t-1\right)}& x^{e-t-1}f(x)&\\
& \ddots & && \vdots  & \vdots &\scriptstyle{e-t}\\
&  &a_{d}&\cdots &a_{t+1}& x^0f(x)& \\
\cline{1-6}
b_{e} &\cdots & & \cdots & b_{t+1-(d-t-1)}&x^{d-t-1}g(x)&\\
&\ddots &&&\vdots & \vdots &\scriptstyle{d-t}\\
&& b_{e} &\cdots & b_{t+1} & x^0g(x)&\\
\cline{1-6} \multicolumn{2}{c}{}
\end{array}
\end{equation*}
with $a_{\ell}=b_{\ell}=0$ for $\ell<0$. When $t=0$ we have $\Sres
_t(f,g)=\Res(f,g)$.

\smallskip In the same article Sylvester also
introduced for $0\le p\le d, 0\le q\le e$ the following {\em
double-sum} expression in $A$ and $B$,

\[
\operatorname*{Sylv}\nolimits^{p,q}(A,B;x)  :=
\sum_{\substack{A^{\prime }\subset A,\,B^{\prime}\subset
B\\|A^{\prime}|=p,\,|B^{\prime}|=q}}R(x,A^{\prime
})\,R(x,B^{\prime})\,\frac{R(A^{\prime},B^{\prime})\,R(A\backslash
A^{\prime},B\backslash B^{\prime})}{R(A^{\prime},A\backslash
A^{\prime })\,R(B^{\prime},B\backslash B^{\prime})},\] where by
convention $R(A',B')=1$ if $A'=\emptyset$ or $B'=\emptyset$. For
instance
\begin{equation}\label{sylv0}\operatorname*{Sylv}\nolimits^{0,0}(A,B;x)=
R(A,B)=\prod_{1\le i\le d,1\le j\le e}(\alpha_i-\beta_j)=
\Res(f,g).\end{equation} We note that
$\operatorname*{Sylv}\nolimits^{p,q}(A,B;x) $ only makes sense when
$\alpha_i\ne \alpha_j$ and $\beta_i\ne \beta_j$ for $i\ne j$, since
otherwise some denominators in
$\operatorname*{Sylv}\nolimits^{p,q}(A,B;x)$ would vanish.
%For convenience, we extend this definition to any $p$ and $q$, setting
%$$\operatorname*{Sylv}\nolimits^{p,q}(A,B;x) :=0 \ \mbox{ when } p<0
%\mbox{ or } p>m \mbox{ or } q<0 \mbox{ or } q>n.$$

\smallskip

The following relation between these double sums and the
subresultants (for monic polynomials with simple roots $f$ and $g$)
was described by Sylvester:  for any choice of $0\le p\le d$ and
$0\le q\le e$ such that $t:=p+q$ satisfies $t<d\le e$ or $t=d<e$,
one has

\begin{equation}\label{sylvester}\Sres_t(f,g)= (-1)^{p(d-t)}{t\choose
p}^{-1}\operatorname*{Sylv}\nolimits^{p,q}(A,B;x)  . \end{equation}

This gives an expression for the subresultant in terms of the
differences of the roots ---generalizing the well-known formula
\eqref{sylv0}--- {\em in case $f$ and $g$ have only simple roots}.
However, when the roots are packed, i.e. when we deal with $\ov A$
and $\ov B$, the expression for the resultant is stable, i.e.
$$\Res(\overline{f},\overline{g})
=\prod_{1\leq i\leq m,\,1\leq j\leq n}(\alpha_i-\beta_j)^{d_ie_j},$$
while not only there is no simple expression of what $\Sres_t(\ov
f,\ov g)$ is in terms of differences of roots but moreover there is
no simple definition of what
$\operatorname*{Sylv}\nolimits^{p,q}(\ov A,\ov B;x)$ should be in
order to preserve Identity~(\ref{sylvester}). Of course, since
$\Sres_t(\ov f,\ov g)$ is defined anyway,
$\operatorname*{Sylv}\nolimits^{p,q}(\ov A,\ov B;x)$ could be
defined as the result
$$\operatorname*{Sylv}\nolimits^{p,q}(\ov A,\ov B;x):= (-1)^{p(d-t)}{t\choose
p} \Sres_t(\ov f,\ov g)$$ but this is not quite satisfactory because
on one hand this does not clarify how $\Sres_t$ behaves in terms of
the roots when these are packed, and on the other hand,
$\operatorname*{Sylv}\nolimits^{p,q}(\ov A,\ov B;x)$ is defined for
every $0\le p\le d$ and $ 0\le q\le e$ while $\Sres_t$ is only
defined for $t:=p+q\le \min\{d,e\}$.

\smallskip
In what follows we express some particular cases
of the subresultant of  two univariate polynomials  in terms of the roots of the polynomials, when these
polynomials have multiple roots. These are partial answers to
the questions raised above, since we were not able to give a right
expression for what the Sylvester double sums should be, even in the
particular cases we could consider. Nevertheless the results we
obtained give a hint of how complex it can be to give complete
general answers, at least in terms of double or multiple sums, see
Theorem~\ref{d=1} below.

\subsection{Generalized Vandermonde and Wronskian matrices}
We need to recall some facts on generalized Vandermonde and
Wronskian matrices.
\begin{notation}\label{gVm} Set $u\in \N$. The {\em generalized Vandermonde}
or {\em confluent} (non-necessarily square) $u\times d$ matrix
$V_u(\overline A)$ associated to $\overline
A=\big((\alpha_1,d_1);\dots;(\alpha_m,d_m)\big) $, \cite{Kal1984},
is
$$
V_u(\overline
A)=V_u\big((\alpha_1,d_1);\ldots;(\alpha_m,d_m)\big):=\begin{array}{|c|c|c|c}
\multicolumn{3}{c}{\scriptstyle{d}}\\
\cline{1-3} & & & \\ V_u(\alpha_1,d_1) & \dots &V_u(\alpha_m,d_m)
 &\scriptstyle u\\ & & & \\
 \cline{1-3}
 \multicolumn{2}{c}{}
\end{array},
$$
where
$$
V_u(\alpha_i,d_i):=\begin{array}{|ccccc|c}
\multicolumn{5}{c}{\scriptstyle{d_i}}\\
\cline{1-5}
1&0 &0&\dots & 0 &\\
\alpha_i&1 &0&\dots & 0 &\\
\alpha^2_i&2\alpha_i &1&\dots & 0 &\scriptstyle u\\
\vdots&  \vdots& \vdots & &\vdots &\\
\alpha_i^{u-1}&(u-1)\alpha_i^{u-2} &{u-1\choose 2}\alpha_i^{u-3}&\dots & {u-1\choose d_i-1}\alpha_i^{u-d_i}& \\
 \cline{1-5}
 \multicolumn{2}{c}{}
\end{array}
$$
with the convention that when $k<j$, ${k\choose j}\alpha_i^{k-j}=0$.

\smallskip \noindent
When $d_i=1$ for all $i$, this gives the usual Vandermonde matrix
$V_u(A)$. When $u=d$, we omit the sub-index $u$ and write
$V(\overline A)$ and $V(A)$.
\end{notation}

For example
$$V\big((\alpha,3);(\beta,2)\big)=\left[\begin{array}{ccc|cc}
1&0&0&1&0\\
\alpha&1&0&\beta&1\\
\alpha^2&2\alpha&1&\beta^2&2\beta\\
\alpha^3&3\alpha^2&3\alpha&\beta^3&3\beta^2\\
\alpha^4&4\alpha^3& 6\alpha^2&\beta^4& 4\beta^3
                       \end{array}
\right]$$ and
$$V_3((\alpha,3);(\beta,2))=\left[\begin{array}{ccccc}
1&0&0&1&0\\
\alpha&1&0&\beta&1\\
\alpha^2&2\alpha&1&\beta^2&2\beta                    \end{array}
\right].$$

\smallskip The determinant of a square confluent matrix is non-zero,
and satisfies, \cite{Aitken},

\begin{equation*}\label{gVd}
\det\big(V(\overline A)\big)=\prod_{1\leq i<j\leq m}
(\alpha_j-\alpha_i)^{d_id_j}.
\end{equation*}

 In the same way that the usual Vandermonde matrix
$V(A)$ is related to the Lagrange Interpolation Problem on $A$, the
generalized Vandermonde matrix $V(\ov A)$ is associated with the
Hermite Interpolation Problem on $\ov A$ \cite{Kal1984}: Given
$\{y_{i,j_i},1\le i\le m, 0\le j_i< d_i\}$, there exists a unique
polynomial $p$ of degree $\deg(p)<d$ which satisfies the following
conditions: {\small
$$\left\{\begin{array}{llll}p(\alpha_1)=0!\,y_{1,0},&
p'(\alpha_1)=1!\,y_{1,1}, & \dots \ , &
  p^{(d_1-1)}(\alpha_1)=(d_1-1)!\,y_{1,d_1-1},\\
\ \ \ \vdots & \ \ \ \vdots  &\ \  \vdots &\ \ \ \vdots \\
p(\alpha_m)=0!\,y_{m,0},&  p'(\alpha_m)=1!\,y_{m,1}, & \dots \ , &
p^{(d_m-1)}(\alpha_m)=(d_m-1)!\,y_{m,d_m-1}.\end{array}\right.$$}
This Hermite polynomial $p=a_0+a_1x+\cdots +a_{d-1}x^{d-1}$ is given
by the only solution of
$$ (a_0 \ a_1 \,\dots \,a_{d-1}) \cdot V(\ov A) = (y_{1,0}\ y_{1,1} \, \dots\, y_{m,d_m-1}) $$
(here the right vector is indexed by the pairs $(i,j_i)$ for $1\le
i\le m, 0\le j_i< d_i$) and satisfies
\begin{equation}\label{hhermite}
\det\big(V(\ov A)\big) \,p(x)=-\,
\det\begin{array}{|cccc|c|l}\multicolumn{1}{c}{}&\multicolumn{1}{c}{}&
\multicolumn{1}{c}{\scriptstyle{d}}&\multicolumn{1}{c}{}&\multicolumn{1}{c}{
\scriptstyle{1}}&
\\\cline{1-5}
&&&&1&\\
&&&&x&\\
&&V(\ov A)&&\vdots&{\scriptstyle{d}}\\
&&&&x^{d-1}&\\
\cline{1-5}  y_{1,0}&y_{1,1}&\ldots&y_{m, d_m-1}&0&{\scriptstyle{1}}
                     \\ \cline{1-5}  \multicolumn{2}{c}{} \end{array}\
 .
\end{equation}

The polynomial  $p$ can also be viewed in a more suitable basis,
where the corresponding ``Vandermonde" matrix has more structure. We
introduce the $d$ polynomials in this basis.

\begin{notation} \label{fiki}
For $ 1\le i\le m$  we set
$$\ov f_i:= \prod_{j\ne i}
(x-\alpha_j)^{d_j}$$ and, for  $  0\le k_i< d_i$,
$$\ov f_{i,k_i}:= \frac{\ov f}{(x-\alpha_i)^{d_i-k_i}}\ = \
(x-\alpha_i)^{k_i}\ov f_i.$$
\end{notation}
%A simple computation, applying  Leibnitz rule for the derivative of
%a product of polynomials:
%$$(fg)^{(q)}=
%\sum_{j=0}^{q}{q\choose j}f^{(j)}\,g^{(q-j)},$$ gives the following
%formula for the  derivatives of $f_{i,k}$, $1\le i\le m$,   $0\le
%k<d_i$, specialized into $\alpha_\ell$:%%%%%
%
%$$ \overline{f}_{i,k}^{(q)}(\alpha_j)=\left\{\begin{array}{cll}0
%&\mbox{for}& j\ne i \ \mbox{ and } \ 0\le q<d_j,\\
%0 &\mbox{for}& j= i \ \mbox{ and } \ 0\le q<k,\\
%\frac{q!}{(q-k)!}\,\ov f_i^{(q-k)}(\alpha_i)&\mbox{for}& j= i \
%\mbox{ and } \ k\le q<d_i.\end{array}\right. $$
%%
%
%This implies that

Then, in this basis, the polynomial $p=\sum_{i,k_i} a_{i,k_i}\ov
f_{i,k_i}$ is given by the only solution of
$$ (a_{1,0} \ a_{1,1} \,\dots \,a_{m,d_m-1}) \cdot V'(\ov A) = (y_{1,0}\ y_{1,1} \, \dots\, y_{m,d_m-1}) $$
where
$$V'(\ov A):=\begin{array}{|c|c|c|c}
\multicolumn{1}{c}{\scriptstyle{d_1}}&\multicolumn{1}{c}{\scriptstyle{}}&
\multicolumn{1}{c}{\scriptstyle{d_m}}&
\\
\cline{1-3} V'(\alpha_1,d_1)& \mathbf{0}& \mathbf{0}&\scriptstyle{d_1} \\
\cline{1-3} \mathbf{0} & \ddots &\mathbf{0}
 & \scriptstyle{}\\ \cline{1-3}\mathbf{0}& \mathbf{0}& V'(\alpha_m,d_m)&\scriptstyle{d_m}  \\
 \cline{1-3}
 \multicolumn{2}{c}{}
\end{array},
$$
with
$$
V'(\alpha_i,d_i):=\begin{array}{|cccc|c}
\multicolumn{4}{c}{\scriptstyle{d_i}}\\
\cline{1-4} \ov f_{i}(\alpha_i)&\ov f_{i}'(\alpha_i) &\dots &
\frac{\ov f_{i}^{(d_i-1)}(\alpha_i)}{(d_i-1)!} &\\
0&\ov f_{i}(\alpha_i)&\dots & \frac{\ov
f_{i}^{(d_i-2)}(\alpha_i)}{(d_i-2)!}  &
\\
0&0 &\ddots & \vdots &\scriptstyle{d_i}\\
\vdots&  \vdots&  &\vdots &\\
0&0 & \dots& \ov f_{i}(\alpha_i)& \\
 \cline{1-4}
 \multicolumn{2}{c}{}
\end{array}
$$
and satisfies
\begin{equation*}\label{hhermitebis}
\det\big(V'(\ov A)\big) \,p(x)=-\,
\det\begin{array}{|cccc|c|l}\multicolumn{1}{c}{}&\multicolumn{1}{c}{}&
\multicolumn{1}{c}{\scriptstyle{d}}&\multicolumn{1}{c}{}&\multicolumn{1}{c}{
\scriptstyle{1}}&
\\\cline{1-5}
&&&&\ov f_{1,0}&\\
&&&&\ov f_{1,1}&\\
&&V'(\ov A)&&\vdots&{\scriptstyle{d}}\\
&&&&\ov f_{m,d_m-1}&\\
\cline{1-5}  y_{1,0}&y_{1,1}&\ldots&y_{m, d_m-1}&0&{\scriptstyle{1}}
                     \\ \cline{1-5} \multicolumn{1}{c}{}\end{array}\
.
\end{equation*}
We note that \begin{multline*}\det(V'(\ov A)) =\prod_{1\le i\le
m}\ov f_i(\alpha_i)^{d_i}\\  = (-1)^{\frac{m(m-1)}{2}}\Big(
\prod_{1\leq i<j\leq m} (\alpha_j-\alpha_i)^{d_id_j}\Big)^2 =
(-1)^{\frac{m(m-1)}{2}}\det(V(\ov A))^2.\end{multline*}

 In particular $p=\sum_{i,j_i}
y_{i,j_i}p_{i,j_i}$ where for $1\le i\le m,\,0\le j_i< d_i$, the
{\em basic Hermite polynomials} $p_{i,j_i}$ are the unique
polynomials of degree $\deg(p_{i,j_i})<d$ determined by the
conditions for $ 1\le \ell\le m, \, 0\le q_\ell< d_\ell$,
\begin{equation}\label{condicion}\left\{\begin{array}{l}
p^{(q_\ell )}_{i,j_i}(\alpha_\ell)=j_i! \ \mbox{ for }  \ \ell=i \mbox{ and } q_\ell=j_i,\\[1mm]
p^{(q_\ell)}_{i,j_i}(\alpha_\ell)=0 \ \mbox{
otherwise}.\end{array}\right.\end{equation}

When  $\ov A=A=\big(\alpha_1,\dots,\alpha_d)$, then, denoting
 $f_i:=\ov f_i$, we have
$$p_{i,0}=\prod_{\ell\ne i}\frac{x-\alpha_\ell}{\alpha_i-\alpha_\ell}=\frac{1}{ f_i(\alpha_i)}\,f_i
\quad \mbox{for} \quad 1\le i\le d,$$ while for $\ov A=(\alpha,d)$,
$$p_{1,j}= (x-\alpha)^{j}=\ov f_{1,j} \quad \mbox{for}\quad  0\le j<
d.$$

The following proposition generalizes these two extremal formulas.

\begin{proposition} \label{piji}
Fix $1\le i\le m$ and $0\le j< d_i$. Then
\begin{equation*}
p_{i,j}= \frac{1}{\ov f_i(\alpha_i)}\,\sum_{k=0}^{d_i-1-j}(-1)^k
\left(\sum_{k_1+\ldots+\widehat{k_i}+\ldots+k_{m}=k} \prod_{\ell\neq
i}\frac{{d_\ell-1+k_\ell\choose k_\ell}}
   {(\alpha_i-\alpha_\ell)^{k_\ell}}\right) \ov f_{i,j+k}
\end{equation*}
where $k_1+\dots +\widehat k_i + \dots + k_m$ denotes the sum
without $k_i$. (When $m=1$,  the right expression under brackets is
understood to equal $1$ for $k=0$ and $0$ otherwise.)
\end{proposition}

  \begin{proof}
  Applying for instance  \cite[Th.~1]{S60}, we first remark that
\begin{equation*}\label{piij}
p_{i,j}= \sum_{k=0}^{d_i-1-j}\frac{1}{k!}\left(\frac{1}{\ov
f_i}\right)^{(k)}(\alpha_i)\,\ov f_{i,j+k}.
\end{equation*}
%verifying that the right-hand side expression satisfies
%Conditions~(\ref{condicion}). By Lemma~\ref{derivatives},   for
%$\ell\ne i$ and $0\le q<d_\ell$ or $\ell=i$ and $0\le q<j$,
%$p_{i,j}^{(q)}(\alpha_\ell)=0$. It remains to show that
%$p^{(j)}_{i,j}(\alpha_i)=j!$ and
%$p^{(q)}_{i,j}(\alpha_i)=0$ for $j<q<d_i$.\\
%But
%$$\ov f_{i,j+k}^{(q)}(\alpha_i)= {q \choose j+k} (j+k)!\,\ov f_i^{(q-j-k)}(\alpha_i) \ \ \mbox{when}
% \ \ j+k\le q,\ \mbox{i.e.  when } \ k\le q-j,
%$$ while $\ov f_{i,j+k}^{(q)}(\alpha_i)= 0$ when $k>q-j$.
%This implies
%\begin{align*}p^{(q)}_{i,j}(\alpha_i)&=\sum_{k=0}^{d_i-1-j}\frac{1}{k!}\left(\frac{1}{\ov
%f_i}\right)^{(k)}(\alpha_i)\,\ov f_{i,j+k}^{(q)}(\alpha_i)\\
%&=\sum_{k=0}^{q-j}\frac{q!}{k!(q-j-k)!}\left(\frac{1}{\ov
%f_i}\right)^{(k)}(\alpha_i)\,\ov f_i^{(q-j-k)}(\alpha_i)\\
%=&\frac{q!}{(q-j)!}\sum_{k=0}^{q-j}{q-j\choose k}\left(\frac{1}{\ov
%f_i}\right)^{(k)}(\alpha_i)\,\ov f_i^{(q-j-k)}(\alpha_i)\\
%=&\frac{q!}{(q-j)!}\left(\frac{\ov f_i}{\ov
%f_i}\right)^{(q-j)}(\alpha_i) \quad = \quad
%\left\{\begin{array}{ccl}j!&\mbox{if}&q=j\\0&\mbox{if}&q>j\end{array}\right..
%\end{align*}
Then we plug into the expression  the following, given by Leibnitz
rule for the derivative of a product:
$$\left(\frac{1}{\ov f_i}\right)^{(k)}(\alpha_i)=
(-1)^k\,k!\,\sum_{k_1+\ldots+\widehat{k_i}+\ldots+k_{r}=k}
\prod_{\ell\neq i}\frac{{d_\ell-1+k_\ell\choose k_\ell}}
   {(\alpha_i-\alpha_\ell)^{d_\ell+k_\ell}}.
$$
\end{proof}

The basic Hermite polynomials enable us to  compute the inverse of
the confluent matrix $V(\overline A)$:
 \begin{equation*}\label{inversa}V(\ov A)^{-1}=\begin{array}{|ccc|c}
\multicolumn{3}{c}{\scriptstyle{d}}\\
\cline{1-3} &V^*_1&& {\scriptstyle{d_1}} \\\cline{1-3} & \vdots &&
 \\ \cline{1-3} &V^*_m& & {\scriptstyle{d_m}}  \\
 \cline{1-3}
 \multicolumn{1}{c}{}
\end{array} \quad\mbox{where} \quad V^*_i:=\begin{array}{|c|c}
\multicolumn{1}{c}{\scriptstyle{d}}\\
\cline{1-1} \mbox{coefficients of } p_{i,1}&  \\ \vdots &
 \scriptstyle d_i\\ \mbox{coefficients of } p_{i,d_i}&  \\
 \cline{1-1}
 \multicolumn{1}{c}{}
\end{array},\ 1\le i\le r
\end{equation*} (here  the coefficients of $p_{i,j_i}(x)$ are written in the monomial
basis $1,x,\ldots x^{d-1}$).

\medskip Now we  set the notation for a slight modification  of a
case of generalized Wronskian matrices.

\begin{notation}\label{gW} Set $u\in \N$.
Given a polynomial $h(z)$,  the {\em generalized Wronskian}
 (non-necessarily square) $u\times d$ matrix $W_{h,u}(\overline A)$
associated to $\overline
A=\big((\alpha_1,d_1);\dots;(\alpha_m,d_m)\big) $ is

{\small
$$
W_{h,u}(\overline
A)=W_{h,u}\big((\alpha_1,d_1);\ldots;(\alpha_m,d_m)\big):=\begin{array}{|c|c|c|c}
\multicolumn{1}{c}{}&\multicolumn{1}{c}{\scriptstyle{d}}& \multicolumn{1}{c}{}& \\
\cline{1-3} & & & \\ W_{h,u}(\alpha_1,d_1) & \dots
&W_{h,u}(\alpha_m,d_m)
 &\scriptstyle u\\ & & & \\
 \cline{1-3}
 \multicolumn{2}{c}{}
\end{array},
$$}
where

$$
W_{h, u}(\alpha_i,d_i):=\begin{array}{|cccc|c}
\multicolumn{4}{c}{\scriptstyle{d_i}}\\
\cline{1-4}
h(\alpha_i)&h'(\alpha_i) &\dots & \frac{h^{(d_i-1)}(\alpha_i)}{(d_i-1)!} &\\
(zh)(\alpha_i)&(zh)'(\alpha_i) &\dots & \frac{(zh)^{(d_i-1)}(\alpha_i)}{(d_i-1)!} &\\
\vdots&  \vdots& &\vdots &\scriptstyle u\\
(z^{u-1}h)(\alpha_i)&(z^{u-1}h)'(\alpha_i) &\dots & \frac{(z^{u-1}h)^{(d_i-1)}(\alpha_i)}{(d_i-1)!} & \\
 \cline{1-4}
 \multicolumn{2}{c}{}
\end{array}\ .
$$
When $u=d$, we omit the sub-index $u$ and write $W_h(\overline A)$.
\end{notation}

For example for $h(z)=x-z$ and $\ov A=(\alpha,3)$,
$$W_{x-z}(\alpha,3)=W_{x-z,3}(\alpha,3)=\begin{array}{|ccc|c}
\multicolumn{4}{c}{\scriptstyle{3}}\\
\cline{1-3}
x-\alpha&-1 &0 & \\
\alpha x-\alpha^2& x-2\alpha&-1&  \scriptstyle 3\\
\alpha^2 x-\alpha^3& 2\alpha x-3\alpha^2& x-3\alpha& \\
 \cline{1-3}
 \multicolumn{2}{c}{}
\end{array}\ .
$$
 The determinant
of a square Wronskian matrix is easily obtainable performing row
operations in the case of one block, and by induction in the size of
the matrix in general:

\begin{equation*}\label{gWd}
\det\big(W_h(\ov A)\big)= \Big(\prod_{1\leq i<j\leq m}
(\alpha_j-\alpha_i)^{d_id_j}\Big) h(\alpha_1)^{d_1}\cdots
h(\alpha_m)^{d_m}.
\end{equation*}

\subsection{Subresultants in multiple roots}
In this section, we describe   explicit formulas  we can get for the
non-trivial  extremal cases of subresultants in terms of both sets
of roots of $\ov f=(x-\alpha_1)^{d_1}\cdots (x-\alpha_m)^{d_m}$ and
$\ov g=(x-\beta_1)^{e_1}\cdots (x-\beta_n)^{e_n}$ with
$d=\sum_{i=1}^m d_i$ and $e=\sum_{j=1}^n e_j$. More precisely, we
present formulas for $\Sres_t(\ov f,\ov g)$  for the cases $t=d-1<e$
(Proposition \ref{t=d-1} below) and $t=1<d\le e$ (Theorem
\ref{d=1}). We will derive them from Theorem~\ref{reph} below, a
generalization of \cite[Th.~3.1]{Hon1999} and
\cite[Lem.~2]{DHKS2007} which includes the multiple roots case (and
is also strongly related to a multiple roots case version of
\cite[Th.~1]{DHKS2009}). The main drawback of this approach to
obtain formulas for all cases of $t$ is the fact that submatrices of
generalized Vandermonde matrices are not always generalized
Vandermonde matrices, so in general their determinants cannot be
expressed as products of differences. This is why the search for
nice formulas in double sums in the case of multiple roots is more
challenging.

\begin{theorem}\label{reph} Set $0\le t\le d< e$ or $0\le t < d=e$.
Then

{ \begin{multline*}  \Sres_t(\ov f,\ov g)    =
(-1)^{d-t}\,\det\big(V(\ov A)\big)^{-1}\,
 \det \
\begin{array}{|c|c|c}
\multicolumn{1}{c}{ \scriptstyle{d}}&\multicolumn{1}{c}{ \scriptstyle{1}}\\
  \cline{1-2}&1&\scriptstyle{}\\
  V_{t+1}(\ov A)& \vdots &\scriptstyle{t+1}\\
  &x^t&\scriptstyle{}\\
\cline{1-2}\  W_{\ov g,d-t}(\ov A)
 &\mathbf{0} &\scriptstyle{d-t}\\
\cline{1-2} \multicolumn{2}{c}{}
\end{array}
\\=   (-1)^{c} \,\det\big(V(\ov A)\big)^{-1} \,\det\big(V(\ov B)\big)^{-1}\, \det
\begin{array}{|c|c|c|c}
\multicolumn{1}{c}{ \scriptstyle{d}}&\multicolumn{1}{c}{ \scriptstyle{e}}&\multicolumn{1}{c}{ \scriptstyle{1}}&\\
\cline{1-3}
&  &1&\\
  V_{t+1}(\ov A)& \mathbf{0} &\vdots &\scriptstyle{t+1}\\
  &  &x^t&\\
\cline{1-3} &  & &\\V_{d+e-t}(\ov A)&V_{d+e-t}(\ov B)&\mathbf{0}&\scriptstyle{d+e-t}\\
&  & &\\
\cline{1-3} \multicolumn{3}{c}{}
\end{array}
, \end{multline*}} where $c:=\max\{e\!\pmod 2,d-t\!\pmod 2\}$.
\end{theorem}

\begin{proof} The proof is quite similar to the proofs of Lemmas 2
and 3 in \cite{DHKS2007}, replacing the usual Vandermonde and
Wronskian matrices by their generalized counterparts. We will thus
omit the intermediate computations.
\\
Let $\ov f=\sum_{i=0}^d a_ix^i$, where $a_d=1$, and $\ov
g=\sum_{j=0}^e b_jx_i$, where $b_e=1$. We introduce the following
matrices of  \cite{DHKS2007}: {\small
\[\begin{array}{ccc}
M_{\ov f}:=
\begin{array}{|ccccc|c}
\multicolumn{5}{c}{\scriptstyle{d+e-t}}\\
\cline{1-5}
a_{0} & \dots & a_{d} &  & &\\
& \ddots &  & \ddots & &\scriptstyle{e-t}\\
&  & a_{0} & \dots & a_{d} & \\
\cline{1-5} \multicolumn{2}{c}{}
\end{array},
& & M_{\ov g}:=
\begin{array}{|ccccc|c}
\multicolumn{5}{c}{\scriptstyle{d+e-t}}\\
\cline{1-5}
b_{0} & \dots & b_{e} &  & &\\
& \ddots &  & \ddots & &\scriptstyle{d-t}\\
&  & b_{0} & \dots & b_{e} & \\
\cline{1-5} \multicolumn{2}{c}{}
\end{array}
\end{array}.
\]}
and {\small $$S_{t}:=%
\begin{array}{|c|c}
\multicolumn{1}{c}{ \scriptstyle{d+e-t}}&\\
\cline{1-1}
\ \ M_{x-z}\ \ &\scriptstyle{t}\\
\cline{1-1}
M_{\ov f}&\scriptstyle{e-t}\\
\cline{1-1}
M_{\ov g}&\scriptstyle{d-t}\\
\cline{1-1} \multicolumn{2}{c}{}
\end{array} \quad \mbox{where} \quad M_{x-z}:=
\begin{array}{|ccccccc|c}
\multicolumn{7}{c}{ \scriptstyle{d+e-t}}\\
\cline{1-7}
x&-1&0& \dots& &\dots&0 &\\
& \ddots &\ddots  & \ddots &&&\vdots &\scriptstyle{t}\\
 & & x & -1&0&\dots & 0 & \\
\cline{1-7} \multicolumn{2}{c}{}
\end{array} .$$}
We have (\cite[Lem.~1]{DHKS2007}):
$$\Sres_{t}(\ov f,\ov g)=(-1)^{(e-t)(d-t)}\det(S_t).$$
Also, exactly as in  the proof of \cite[Lem.~2]{DHKS2007}, {\small
$$
\begin{array}{c|c|}
\multicolumn{1}{c}{}&
\multicolumn{1}{c}{\scriptstyle{d+e-t}}\\
\cline{2-2}
\scriptstyle{t}& \ \ M_{x-z}\ \ \\
\cline{2-2}
\scriptstyle{e-t}& M_{\ov f}\\
\cline{2-2} \scriptstyle{d-t}&
M_{\ov g}\\
\cline{2-2} \multicolumn{2}{c}{}
\end{array}
\,
\begin{array}{|c|c|c}
\multicolumn{1}{c}{ \scriptstyle{d}}&\multicolumn{1}{c}{ \scriptstyle{e-t}}&\\
\cline{1-2} & \mathbf{0} &\scriptstyle{d}\\\cline{2-2}
V_{d+e-t}(\ov A )&\\
&\Id_{e-t} &\scriptstyle{e-t}\\
\cline{1-2} \multicolumn{3}{c}{}
\end{array}\   =    \  \begin{array}{|c|c|c}
\multicolumn{1}{c}{ \scriptstyle{d}}&\multicolumn{1}{c}{ \scriptstyle{e-t}}&\\
\cline{1-2}
  W_{x-z,t}(\ov A)& * &\scriptstyle{t}\\
\cline{1-2} \mathbf{0}&M'_{\ov f}&\scriptstyle{e-t}\\
\cline{1-2}\  W_{\ov g,d-t}(\ov A)
& * &\scriptstyle{d-t}\\
\cline{1-2} \multicolumn{3}{c}{}
\end{array}.
$$}
This implies first the generalization of \cite[Lem.~2]{DHKS2007} to
the multiple roots case: {\small $$\det\big({V}(\ov
A)\big)\,\Sres_t(\ov f,\ov g) = \ \det \
\begin{array}{|c|c}
\multicolumn{1}{c}{ \scriptstyle{d}}&\\\cline{1-1}
  W_{x-z,t}(\ov A)& \scriptstyle{t}\\
\cline{1-1}\  W_{\ov g,d-t}(\ov A)
 &\scriptstyle{d-t}\\
\cline{1-1} \multicolumn{2}{c}{}
\end{array} = (-1)^{d-t} \det \
\begin{array}{|c|c|c}
\multicolumn{1}{c}{ \scriptstyle{d}}&\multicolumn{1}{c}{ \scriptstyle{1}}\\
  \cline{1-2}&1&\scriptstyle{}\\
  V_{t+1}(\ov A)& \vdots &\scriptstyle{t+1}\\
  &x^t&\scriptstyle{}\\
\cline{1-2}\  W_{\ov g,d-t}(\ov A)
 &\mathbf{0} &\scriptstyle{d-t}\\
\cline{1-2} \multicolumn{2}{c}{}
\end{array},$$}
where the second equality is a consequence of obvious   row and
column operations.
% coming from  the row and column equivalences:
%$$\det \,
%\begin{array}{|c|c}
%\multicolumn{1}{c}{ \scriptstyle{d}}&\\\cline{1-1}
%  W_{x-z,t}(\ov A)& \scriptstyle{t}\\
%\cline{1-1}\  W_{\ov g,d-t}(\ov A)
% &\scriptstyle{d-t}\\
%\cline{1-1} \multicolumn{2}{c}{}
%\end{array} = (-1)^{d}\det \,\begin{array}{|c|c|c}
%\multicolumn{1}{c}{ \scriptstyle{d}}&\multicolumn{1}{c}{
%\scriptstyle{1}}& \\ \cline{1-2} \mathbf{0}&1& \scriptstyle{1}\\
%\cline{1-2}
%  W_{z-x,t}(\ov A)& \mathbf{0}& \scriptstyle{t}\\
%\cline{1-2}\  W_{\ov g,d-t}(\ov A)& \mathbf{0}
% &\scriptstyle{d-t}\\
%\cline{1-2} \multicolumn{2}{c}{}
%\end{array}$$
%and
%$$\begin{array}{|c|c|c}
%\multicolumn{1}{c}{ \scriptstyle{d_i}}&\multicolumn{1}{c}{ \scriptstyle{1}}&\\
%\cline{1-2}
% \mathbf{}  &1&\\
%  V_{t+1}(\alpha_i,d_i)& \vdots&\scriptstyle{t+1}\\
%   \mathbf{}  &x^t&\\
%\cline{1-2} \multicolumn{3}{c}{}
%\end{array}\equiv_r
%\begin{array}{|c|c|c}
%\multicolumn{1}{c}{ \scriptstyle{d_i}}&\multicolumn{1}{c}{ \scriptstyle{1}}&\\
%\cline{1-2}
%1 \quad  0\ \cdots \ 0&1&\scriptstyle{1}\\
%\cline{1-2}
%  W_{z-x,t}(\alpha_i,d_i)&\mathbf{0}&\scriptstyle{t}\\
%\cline{1-2} \multicolumn{3}{c}{} \end{array}\equiv_c
%\begin{array}{|c|c|c}
%\multicolumn{1}{c}{ \scriptstyle{d_i}}&\multicolumn{1}{c}{ \scriptstyle{1}}&\\
%\cline{1-2}
%\mathbf{0}&1&\scriptstyle{1}\\
%\cline{1-2}
%  W_{z-x,t}(\alpha_i,d_i)&\mathbf{0}&\scriptstyle{t}\\
%\cline{1-2} \multicolumn{3}{c}{} \end{array}.$$
Next, we get  as in the proof of \cite[Lem.~3]{DHKS2007}, {\small
\begin{align*}&\det\big({V}(\ov A)\big)
\det\big(V(\ov B))\, \Sres_t(\ov f,\ov g)
%=
%(-1)^{d-t + e +(d-t)e}\,\det%
%    \  \begin{array}{|c|c|c|c}
%\multicolumn{1}{c}{ \scriptstyle{d}}&\multicolumn{1}{c}{
%\scriptstyle{e}}&\multicolumn{1}{c}{
%\scriptstyle{1}}&\\
%\cline{1-3} &   &1 & \\[-3mm]
%  V_{t+1}(\ov A)& \mathbf{0} &\vdots & \scriptstyle{t+1}\\[-3mm]
%   &   &x^t & \\
%\cline{1-3} V_e(\ov A)&V_e(\ov B)&\mathbf{0} &\scriptstyle{e}\\
%\cline{1-3}\  W_{\ov g,d-t}(\ov A)
%& \mathbf{0} &\mathbf{0} &\scriptstyle{d-e}\\
%\cline{1-3} \multicolumn{3}{c}{}
%\end{array} \\ &= \ (-1)^{c}\, \det \left(
%\begin{tabular}
%[c]{cccc}
%& $\scriptstyle t+1$ & $\scriptstyle e$ & $\scriptstyle d-t$\\\cline{2-4}%
%$\scriptstyle t+1$ & \multicolumn{1}{|c}{$\Id_{t+1}$} &
%\multicolumn{1}{|c}{$\mathbf{0}$} &
%\multicolumn{1}{|c|}{$\mathbf{0}$}\\\cline{2-4}%
%$\scriptstyle e$ & \multicolumn{1}{|c}{$\mathbf{0}$} &
%\multicolumn{1}{|c}{$\Id_e$} &
%\multicolumn{1}{|c|}{$\mathbf{0}$}\\\cline{2-4}%
%$\scriptstyle d-t$ & \multicolumn{1}{|c}{$\mathbf{0}$} & \multicolumn{2}{|c|}{$M_{\ov g}$}\\\cline{2-4}%
%\end{tabular}
%\,\ \  \begin{array}{|c|c|c|c} \multicolumn{1}{c}{%
%\scriptstyle{d}}&\multicolumn{1}{c}{
%\scriptstyle{e}}&\multicolumn{1}{c}{
%\scriptstyle{1}}&\\
%\cline{1-3} &   &1 & \\[-3mm]
%  V_{t+1}(\ov A)& \mathbf{0} &\vdots & \scriptstyle{t+1}\\[-3mm]
%   &   &x^t & \\
%\cline{1-3} V_{d+e-t}(\ov A)&V_{d+e-t}(\ov B)&\mathbf{0} &\scriptstyle{d+e-t}\\
%\cline{1-3} \multicolumn{3}{c}{}
%\end{array}
%\right)
%\\ &
= \ (-1)^{c}\, \det
 \begin{array}{|c|c|c|c} \multicolumn{1}{c}{
\scriptstyle{d}}&\multicolumn{1}{c}{
\scriptstyle{e}}&\multicolumn{1}{c}{
\scriptstyle{1}}&\\
\cline{1-3} &   &1 & \\[-3mm]
  V_{t+1}(\ov A)& \mathbf{0} &\vdots & \scriptstyle{t+1}\\[-3mm]
   &   &x^t & \\
\cline{1-3} V_{d+e-t}(\ov A)&V_{d+e-t}(\ov B)&\mathbf{0} &\scriptstyle{d+e-t}\\
\cline{1-3} \multicolumn{3}{c}{}
\end{array}
 .
\end{align*}}
%since the first matrix of the second row is lower triangular with
%$1$ in the diagonal.
\end{proof}

 We note that starting from the first equality above
 %$$\det\big({V}(\ov
%A)\big)\,\Sres_t(\ov f,\ov g) = \ \det \
%\begin{array}{|c|c} \multicolumn{1}{c}{
%\scriptstyle{d}}&\\\cline{1-1}
%  W_{x-z,t}(\ov A)& \scriptstyle{t}\\
%\cline{1-1}\  W_{\ov g,d-t}(\ov A)
% &\scriptstyle{d-t}\\
%\cline{1-1} \multicolumn{2}{c}{}
%\end{array}$$
and applying similar arguments, we also get very simply {\small
\begin{equation}\label{sresul} \det\big({V}(\ov A)\big)
\det\big(V(\ov B))\, \Sres_t(\ov f,\ov g) =
(-1)^{(d-t)e}\,\det%
    \  \begin{array}{|c|c|c}
\multicolumn{1}{c}{ \scriptstyle{d}}&\multicolumn{1}{c}{
\scriptstyle{e}}&\\
\cline{1-2}
 W_{x-z}(\ov A)& \mathbf{0} & \scriptstyle{t}\\
\cline{1-2} V_{d+e-t}(\ov A)&V_{d+e-t}(\ov B)&\scriptstyle{d+e-t}\\
\cline{1-2} \multicolumn{3}{c}{}
\end{array}.
\end{equation}}

%\subsection{Extremal subresultants}

As mentioned above, when $t=0$ the formula in roots for
$\Sres_0(f,g)$ specializes well when considering $\Sres_0(\ov f,\ov
g)$. When $t=d<e$, the formula $\Sres_d(f,g)= \prod_{1\le i\le
d}(x-\alpha_i)$ also specializes well as $\Sres_d(\ov f,\ov g)=
\prod_{1\le i\le m}(x-\alpha_i)^{d_i}$.
 Our purpose now is to understand
formulas in roots for the following extremal subresultants, i.e  for
$\mbox{Sres}_1$ and $\mbox{Sres}_{d-1}$, in case  of multiple roots.

\smallskip
\noindent
 $\bullet$ The case $t=d-1<e$:
When $f$ has simple roots,  it is known (or can easily be derived
for instance from Sylvester's
 Identity~(\ref{sylvester}) for
 $p=d-1$ and $q=0$)
 that
$$\operatorname*{Sres}\nolimits_{d-1}(f,\ov g)= \sum_{ i=1}^{ d} \ov g(\alpha_i)\,\Big(\prod_{j\ne i}
\frac{x-\alpha_j}{\alpha_i-\alpha_j}\Big)=\sum_{ i=1}^{ d} \ov
g(\alpha_i)\,p_i,$$ where $p_i$ is the basic Lagrange interpolation
polynomial of degree strictly bounded by $d$ such that
$p_i(\alpha_i)=1$ and $p_i(\alpha_j)=0$ for $j\ne i$. In other
words, $\operatorname*{Sres}\nolimits_{d-1}(f,\ov g)$ is the
Lagrange interpolation polynomial of degree strictly bounded by $d$
which coincides with $\ov g$ in the $d$ values
$\alpha_1,\dots,\alpha_d$. This formula does not apply when $f$ has
multiple roots, but  we can show that we  get the natural
generalization of this fact, that is, that
$\operatorname*{Sres}\nolimits_{d-1}(\ov f,\ov g)$ is the Hermite
interpolation polynomial of degree strictly bounded by $d$ which
coincides with $\ov g$  and its derivatives up to the corresponding
orders  in the $m$ values $\alpha_1,\dots,\alpha_m$:

\begin{proposition}\label{t=d-1}$$\operatorname*{Sres}\nolimits_{d-1}(\ov f,\ov g)= \sum_{i=1}^m\sum_{j_i=0}^{d_i-1}
\frac{\overline{g}^{(j_i)}(\alpha_i)}{j_i!}\,p_{i,j_i},$$ where
$p_{i,j_i}$ is the basic Hermite interpolation polynomial defined by Condition~(\ref{condicion}) or
Proposition~\ref{piji} for $\ov A$.
\end{proposition}

\begin{proof}

In this  case, applying the first statement of Theorem~\ref{reph}
we get
$$\Sres_{d-1}(\ov f,\ov g)  =
-\,\det\big(V(\ov A)\big)^{-1}\,
 \det \
\begin{array}{|c|c|c}
\multicolumn{1}{c}{ \scriptstyle{d}}&\multicolumn{1}{c}{ \scriptstyle{1}}\\
  \cline{1-2}&1&\scriptstyle{}\\
  V_{d}(\ov A)& \vdots &\scriptstyle{d}\\
  &x^{d-1}&\scriptstyle{}\\
\cline{1-2}\  W_{\ov g,1}(\ov A)
 &\mathbf{0} &\scriptstyle{1}\\
\cline{1-2} \multicolumn{2}{c}{}
\end{array}$$
where when following the subindex notation of
Formula~(\ref{hhermite}), we note that
$$\big(W_{\ov g,1}(\ov A)\big)_{i,j_i}=\frac{\ov
g^{(j_i)}(\alpha_i)}{j_i!}.
$$
The conclusion follows by Formula~(\ref{hhermite}).
\end{proof} For example, when $\ov A=(\alpha,d)$, we get
$$\operatorname*{Sres}\nolimits_{d-1}((x-\alpha)^d,\ov g)=
\sum_{j=0}^{d-1}
\frac{\overline{g}^{(j)}(\alpha)}{j!}\,(x-\alpha)^j,$$ the Taylor
expansion of $\ov g$ up to order $d-1$.

\medskip \noindent
 $\bullet$ The case $t=1<d\le e$:
We keep Notation~\ref{fiki}. When $f$ has simple roots,  it is known
(or can easily be derived for instance from Sylvester's
 Identity~(\ref{sylvester}) for
 $p=1$ and $q=0$)
 that
\begin{align}\label{van} \operatorname*{Sres}\nolimits_{1}(f,\ov g)
 &=(-1)^{d-1}\sum_{i=1}^d\Big( \prod_{j\ne i} \frac{\ov g(\alpha_j)}{\alpha_i-\alpha_j}\Big)
 (x-\alpha_i)\\ \nonumber &=(-1)^{d-1}\sum_{i=1}^d \frac{\prod_{j\ne i} \ov g(\alpha_j)}{f_i(\alpha_i)}\,(x-\alpha_i).
\end{align}
The general situation is a bit less obvious, but in any case we can
get an expression of $\mbox{Sres}_1(\overline{f},\overline{g})$ by
using the coefficients of the Hermite interpolation
 polynomial,  in this case of the whole data  $$\ov A \cup \ov
B:=\big(
(\alpha_1,d_1);\dots;(\alpha_m,d_m);(\beta_1,e_1);\dots;(\beta_n,e_n)\big).$$
We note that  $$\det\big(V(\ov A\cup \ov B)\big)=\det\big(V(\ov
A)\big)\det\big(V(\ov B)\big)R(\ov B,\ov A)$$ which holds even when
$\alpha_i=\beta_j$ for some $i,j$.

 \begin{theorem}\label{d=1}{\small
\begin{align*}\Sres_1(\overline{f},\overline{g})&=\sum_{i=1}^{
m}(-1)^{d-d_i} \Big(\frac{\prod_{j\neq
i} \ov g(\alpha_j)^{d_j}}{\ov
f_i(\alpha_i)}\Big)\ov g(\alpha_i)^{d_i-1}
  \Big((x-\alpha_i)\cdot \\ &\sum_{\tiny\begin{array}{c}k_1+\cdots
  +\widehat k_i+\cdots \\ \cdots +k_{m+n}=d_i-1\end{array}}\prod_{\tiny \begin{array}{c}1\le j\le m\\j\neq
i\end{array}}
  \frac{{d_j-1+k_j\choose
  k_j}}{(\alpha_i-\alpha_j)^{k_j}}\prod_{1\le \ell\le n}\frac{{e_\ell-1+k_{m+\ell}\choose
  k_{m+\ell}}}{
  (\alpha_i-\beta_\ell)^{k_{m+\ell}}}\\
  &+ \min\{1,d_i-1\} \sum_{\tiny\begin{array}{c}k_1+\cdots
  +\widehat k_i+\cdots \\ \cdots +k_{m+n}=d_i-2\end{array}}\prod_{\tiny \begin{array}{c}1\le j\le m\\j\neq
i\end{array}}
  \frac{{d_j-1+k_j\choose
  k_j}}{(\alpha_i-\alpha_j)^{k_j}}\prod_{1\le \ell\le n}\frac{{e_\ell-1+k_{m+\ell}\choose
  k_{m+\ell}}}{
  (\alpha_i-\beta_\ell)^{k_{m+\ell}}}\Big).
  \end{align*}}
  \end{theorem}

\begin{proof}
 Setting $t=1$ in Expression~(\ref{sresul})  we get
$$ \det\big({V}(\ov A)\big)
\det\big(V(\ov B)\big)\, \Sres_1(\ov f,\ov g)=\ (-1)^{(d-1)e}\, \det
 \begin{array}{|c|c|c}
\multicolumn{1}{c}{ \scriptstyle{d}}&\multicolumn{1}{c}{ \scriptstyle{e}}&\\
\cline{1-2}
  W_{x-z,1}(\ov A)& \mathbf{0} &\scriptstyle{1}\\
\cline{1-2} V_{d+e-1}(\ov A)&V_{d+e-1}(\ov B)&\scriptstyle{d+e-1}\\
\cline{1-2} \multicolumn{3}{c}{}
\end{array}$$
where $$W_{x-z,1}=(\underbrace{x-\alpha_1, -1, 0,\dots,
0}_{\scriptstyle_{d_1}},\dots, \underbrace{x-\alpha_m,-1,0
\dots,0}_{\scriptstyle d_m}).$$ We  expand the determinant w.r.t.
the first row, and observe that when we delete the first row and
column $j$, the matrix that survives coincides with $V(\ov A \cup
\ov B)_{(d+e,j)}$, the submatrix of $V(\ov A\cup \ov B)$ obtained by
deleting the last row and column $j$. Therefore,
\begin{align*} & \ \det
\begin{array}{|c|c|}
 \cline{1-2}
  W_{x-z,1}(\ov A)& \mathbf{0} \\
\cline{1-2} V_{d+e-1}(\ov A)&V_{d+e-1}(\ov B)\\
\cline{1-2}
\end{array}\\ & \quad =\sum_{j=1}^m (-1)^{\phi(j)-1}\Big(\det\big(V(\ov A\cup
\ov B)|_{(d+e,\phi(j))}\big)(x-\alpha_j)+ \det\big(V(\ov A\cup \ov
B)|_{(d+e,\phi'(j))}\big)\Big),
\end{align*}
where $\phi(i)$ equals the number of the column  corresponding to
$(1,\alpha_i,\dots,\alpha_i^{d+e-1})$ in
 $V(\overline{A}\cup\overline{B})$, and
 $\phi'(i)=\phi(i)+1$ if $d_i>1$ and $0$ otherwise.

\smallskip
\noindent Now, from
\begin{align*}&\det\big(V(\ov A\cup \ov B)|_{(d+e,\phi(j))}\big)
=(-1)^{d+e-\phi(j)}\,\det\big(V(\ov A\cup \ov B)\big) \, V(\ov A
\cup \ov B)^{-1}_{\phi(j),d+e},\\& \det\big(V(\ov A\cup \ov
B)|_{(d+e,\phi'(j))}\big) =(-1)^{d+e-\phi'(j)}\,\det\big(V(\ov A\cup
\ov B)\big) \, V(\ov A \cup \ov B)^{-1}_{\phi'(j),d+e} \end{align*}
(by the cofactor expression for the inverse) and $$\det\big(V(\ov
A\cup \ov B)\big)=(-1)^{de}\det \big(V(\ov A)\big)\det \big(V(\ov
B)\big)R(\ov A,\ov B)$$ we first get, since $ R(\ov A,\ov
B)=\prod_{1\le i\le m} \ov g(\alpha_i)^{d_i}$, that
{\small\begin{equation*}\label{interm}\Sres_1(\overline{f},\overline{g})=
 (-1)^{d-1}\big(\prod_{1\le i\le m}
\ov g(\alpha_i)^{d_i}\big)\left(\sum_{i=1}^m V(\ov A\cup \ov B)^{-1}
_{\phi(i),d+e}(x-\alpha_i)-
 \sum_{i=1}^m V(\ov A \cup \ov
 B)^{-1}_{\phi'(i),d+e}\right).\end{equation*}}
\noindent We set $\ov h:=\overline{f}\,\overline{g}$, and for
$i=1,\ldots,m$, $\ov h_i:=\ov h/(x-\alpha_i)^{d_i}$. In
\cite[Id.~9]{Cs75}, it is shown that
 $$V(\ov A\cup\ov B)^{-1}
_{\phi(i),d+e}=\frac{1}{(d_i-1)!}\left(\frac{1}{\ov
h_i}\right)^{(d_i-1)}(\alpha_i),
 $$
 and when $d_i>1$,
 $$V(\ov A\cup\ov B)^{-1}_{\phi'(i),d+e}
=\frac{1}{(d_i-2)!}\left(\frac{1}{\ov
h_i}\right)^{(d_i-2)}(\alpha_i).
 $$
 Therefore, we obtain the statement by applying
Leibnitz rule  {\small
$$\left(\frac{1}{\ov h_i}\right)^{(k)}=(-1)^k \,k!\,\sum_{k_1+\dots+\widehat{k_i}+
\cdots +k_{m+n}=k } \prod_{\tiny \begin{array}{c}1\le j\le m\\j\neq
i\end{array}}\frac{{d_j-1+k_j\choose k_j}}
   {(x-\alpha_j)^{d_j+k_j}}\prod_{1\le \ell\le n}\frac{{e_\ell-1+k_{m+\ell}\choose k_{m+\ell}}}
   {(x-\beta_\ell)^{e_\ell+k_{m+\ell}}}.
$$}
\end{proof}
Note that in the case that $f$ has simple roots we immediately
recover Identity~(\ref{van}) while when $\ov f=(x-\alpha)^d$ for
$d\ge 2$, we recover Proposition~3.2 of \cite{DKS09}:
\begin{eqnarray*}\Sres_{1}((x-\alpha)^d,\ov g)&=&
\ov g(\alpha)^{d-1} \Big({ \sum_{k_1+\dots+k_n=d-1} \Big(\prod_{\ell=1}^n
\frac{{e_\ell-1+k_\ell\choose
k_\ell}}{(\alpha-\beta_\ell)^{k_\ell}}\Big)(x-\alpha)}+\\
& &{\sum_{k_1+\cdots +k_n=d-2}\ \prod_{\ell=1}^n
 \frac{{e_\ell-1+k_\ell\choose k_\ell}}{(\alpha-\beta_\ell)^{k_\ell}}  }\ \Big)
.\end{eqnarray*}

\bigskip
\section{Multivariate Case: Poisson-like formulas for Subresultants}\label{3}
We turn to the multivariate case, considering the definition of
subresultants introduced in  \cite{Cha}.  Our goal is to generalize
Theorem $3.2$ in \cite{DKS2006} --that we recall below-- to the case when the
considered polynomials have multiple roots. We first  fix the notation, referring  the reader to \cite{DKS2006} for more
details.
\subsection{Notation}
Fix  $n\in \N$ and set $D_i\in \N$ for  $1\le i\le n+1$. Let
$$f_i:=\sum_{|\bfalpha|\le D_i}a_{i,\bfalpha}{\bfx}^\bfalpha \ \in \ K[x_1,\dots,x_n],$$
be  polynomials of degree $D_i$ in $n$ variables,  where
$\bfalpha=(\alpha_1,\dots,\alpha_n)\in \left(\Z_{\geq0}\right)^n$,
${\bfx}^\bfalpha:=x_1^{\alpha_1}\cdots x_n^{\alpha_n}$,
$|\bfalpha|=\alpha_1+\cdots + \alpha_n$, and $K$ is a field of
characteristic zero, that we assume without loss of generality to be
algebraically closed.

\smallskip Fix $t\in \N$. Let $k:=\HH_{D_1\dots D_{n+1}}(t)$ be the
Hilbert function at $t$ of a regular sequence of $n+1$ homogeneous
polynomials in $n+1$ variables of degrees $D_1,\dots,D_{n+1}$, i.e.
$$k=\#\{\bfx^\bfalpha: |\bfalpha|\le t, \, \alpha_i<D_i,\,  1\le i\le n,
\mbox{ and } t-|\bfalpha|<D_{n+1}\}.$$ We set
 $$
 \SS:=\{\bfx^{\bfgamma_1},\ldots,\bfx^{\bfgamma_k}\}\subset K[\bfx]_t
 $$
 a set of $k$ monomials of degree bounded by $t$, and
 $$\Delta_{\SS}(f_1, \ldots, f_{n+1}):=\Delta^{(t)}_{\SS^h}(\fh_1,\dots,\fh_{n+1}),$$
for the {\em order $t$  subresultant of $\fh_1,\dots,\fh_{n+1}$
with respect to the family}
$\SS^h:=\{\bfx^{\bfgamma_1}x_{n+1}^{t-|\bfgamma_1|},\ldots,\bfx^{\bfgamma_k}x_{n+1}^{t-|\bfgamma_k|}\}
$ defined in \cite{Cha}.   Here, $\fh_i$ denotes the homogenization of $f_i$ by the
 variable $x_{n+1}$.

\smallskip

 We recall that the {subresultant} $\Delta_\SS$ is a polynomial in
 the coefficients of the $f^h_i$ of degree
$\HH_{D_1\dots D_{i-1} D_{i+1} \dots D_{n+1}}(t-D_i)$ for $1\le i\le
n+1$,
 having the following  property:  $\Delta_\SS=0$  if and only if $I_t\cup  \SS^h$ does not
generate the space of all forms of degree $t$ in $k[x_1, \ldots,
x_{n+1}]$, where $I_t$ denotes the degree $t$ part of the ideal generated
by the  $\fh_i$'s.

\smallskip By \cite{Cha2} we know that
\begin{eqnarray}\label{Extra}
\det(M_\SS)= \EE(t)\,\Delta_{\SS},
 \end{eqnarray}
where  $ M_\SS$
%\in K^{(N-k)\times (N-k)}$
denotes the Macaulay-Chardin matrix obtained
from
\begin{equation}\label{sylv}
\left[\begin{array}{c}\;\;M_{f_1}\;\;\\[-2mm]\vdots \\[-2mm]\;\;M_{f_{n+1}}\;\;\end{array}\right]
\end{equation}
by deleting the  columns indexed by the monomials in $\SS$, and
$\EE(t)$ is the extraneous factor defined as the determinant of a specific
square submatrix of (\ref{sylv})
%whose rows and columns correspond
%to monomials  $\bfx^\bfalpha$ having at least two different indexes
%$1\leq i,j\leq n$ satisfying $\alpha_i\geq d_i,\ \alpha_j\geq d_j$
(see \cite{Cha,Cha2,DKS2006}).

\smallskip \noindent
 We set $\rho:=(D_1-1)+\cdots +
(D_n-1)$ and for $j\ge 0$,    $\tau_j:= \HH_{D_1\ldots D_n}(j)$, the
Hilbert function at $j$ of a regular sequence of $n$ homogeneous
polynomials in $n$ variables of degrees $D_1,\ldots , D_n$. We define also
\begin{equation}\label{prof}
\TT_{j}:=\left\{\small\begin{array}{l} \mbox{\textit{any} set of }
\tau_j \mbox{ monomials of degree } j
 \ \mbox{ if  }  j\ge \max\{0,t-D_{n+1}+1\}
\\[1mm]
\{\bfx^\bfalpha: \,|\bfalpha|=j, \alpha_i<D_i \mbox{ for } 1\le i\le
n\}\ \mbox{ if } 0\le j<t-D_{n+1}+1,\end{array}\right.
\end{equation}
and denote with $D:=D_1\cdots D_n$  the {\em B\'ezout number}, the number of
common solutions in ${K}^n$ of $n$ generic polynomials.\\
Set $ \TT:=\cup_{j\geq0}\TT_j$ and
$\TT^*:=\cup_{j=t+1}^{\rho}\TT_j$. Note
  that $|\TT|={D}.$ We enumerate the elements of $\TT$ as follows:
$\TT=\{\bfx^{\bfalpha_1},\ldots,\bfx^{\bfalpha_{D}}\},$ and assume
that for $s:= |\TT^*|$ we have  $
\TT^*=\{\bfx^{\bfalpha_1},\ldots,\bfx^{\bfalpha_{s}}\}$ . Also set
\begin{equation}\label{R} {\mathcal
R}:= \{\bfx^{\bfbeta_1},\dots,\bfx^{\bfbeta_r}\} =
\{\bfx^\bfalpha:\,|\bfalpha|\le t, \alpha_i<D_i, 1\le i\le n,
 \ t-|\bfalpha|\geq D_{n+1}\}.\end{equation} Finally, for $1\le i \le n$, let $\widetilde f_i$
be the homogeneous component
 of degree $D_i$ of $f_i$,
  and
 $\widetilde\Delta_{\TT_j}:=\Delta_{\TT_{j}}^{(j)}(\widetilde f_1,\dots,\widetilde f_n)$
 be the order $j$ subresultant of $\widetilde f_1,\dots,
 \widetilde f_n$  with respect to $\TT_{j}$.

\subsection{Poisson-like formula for subresultants}

From now on we assume that $f_1,\dots, f_{n}$ are generic in the
sense they have no roots at infinity (which implies by B\'ezout
theorem that the quotient algebra $A:=K[\bfx]/(f_1,\dots,f_n)$ is a
finitely dimensional $K$-vector space of dimension $D$, which equals
the number of common roots in $K^n$ of these polynomials, counted
with multiplicity, see e.g. \cite[Ch. 3, Th. 5.5]{CLO98}), and that
$\TT$ is a basis of $A$.

 In
\cite{DKS2006} we treated the case of general polynomials with
indeterminate coefficients, which specializes well under our
assumptions to the  case when the  common  roots ${\bfxi}_1,\dots,
{\bfxi}_{D}$  of $f_1,\dots, f_n$ in $ K^n$  are all {\em simple}.
Set  $Z:=\{{\bfxi}_1,\dots, {\bfxi}_{D} \}$. We introduced the
Vandermonde matrix
\begin{equation}\label{V_T}
V_\TT(Z):=\begin{array}{|ccc|}\cline{1-3}
{\bfxi}_1^{\bfalpha_1} &\cdots& {\bfxi}_{D}^{\bfalpha_1}\\%[-2mm]
\vdots & & \vdots \\
{\bfxi}_1^{\bfalpha_D} &\cdots& {\bfxi}_{D}^{\bfalpha_D}\\
\cline{1-3}
\end{array}\ \in \ {K}^{D\times D}
\end{equation}
whose determinant is non zero, since  $\TT$ is assumed to be a  basis of $A$, and we
defined {\small
\begin{equation}\label{OS} \OO_\SS(Z):=\begin{array}{|ccc|c}
\multicolumn{1}{c}{}&\multicolumn{1}{c}{\scriptstyle D}&
\multicolumn{1}{c}{}&
\\
\cline{1-3}
{\bfxi}_1^{\bfgamma_1} &\cdots& {\bfxi}_{D}^{\bfgamma_1}&\\%[-2mm]
\vdots & & \vdots& \scriptstyle{k}\\%[-2mm]
{\bfxi}_1^{\bfgamma_k} &\cdots& {\bfxi}_{D}^{\bfgamma_k}\\%[-2mm]
\cline{1-3}
{\bfxi}_1^{\bfalpha_1} &\cdots& {\bfxi}_{D}^{\bfalpha_1}&\\%[-2mm]
\vdots & & \vdots &\scriptstyle s\\%[-2mm]
{\bfxi}_1^{\bfalpha_s} &\cdots& {\bfxi}_{D}^{\bfalpha_s}\\%[-2mm]
\cline{1-3} {\bfxi}_1^{\bfbeta_1}f_{n+1}({\bfxi}_1)&\cdots&
{\bfxi}_{D}^{\bfbeta_1}f_{n+1}({\bfxi}_{D})& \\%[-2mm]
\vdots& &\vdots& \scriptstyle r\\%[-2mm]
{\bfxi}_1^{\bfbeta_r}f_{n+1}({\bfxi}_1)& \cdots&
{\bfxi}_{D}^{\bfbeta_r}f_{n+1}({\bfxi}_{D})&\\ \cline{1-3}
\multicolumn{2}{c}{}
\end{array}\ \in \ {K}^{D\times D}.
\end{equation}}

\begin{theorem}\label{subresultant2}\cite[Th.~3.2]{DKS2006}
For any  $t\in\Z_{\geq0}$ and for any
$\SS=\{\bfx^{\bfgamma_1},\ldots,\bfx^{\bfgamma_k}\}$ $  \subset
K[\bfx]_t$ of cardinality  $k=\HH_{D_1\dots D_{n+1}}(t)$, we have
\begin{equation*}\label{quotient}
\Delta_\SS(f_1,\dots,f_{n+1})=\pm \left(\prod_{j=t-D_{n+1}+1}^t
\widetilde \Delta_{\TT_{j}}\right) \frac{\det
\big(\OO_\SS(Z)\big)}{\det \big(V_{\TT}(Z)\big)}.
\end{equation*}
\end{theorem}

\smallskip
In order to generalize this result to  systems with multiple roots,
and obtain an expression for  the subresultant in terms of the roots
of the first $n$ polynomials $f_1, \ldots, f_n$, we need to
introduce notions of the multiplicity structure of the roots  that
are sufficient to define $( f_1, \ldots, f_n) $. To be more precise,
in the case of multiple roots, the set of  evaluation maps
$\{\mbox{ev}_{\bfxi}:A\to K\;| \  \bfxi\ \mbox{ common root of }
f_1,\dots,f_n\} $   is not anymore a basis
 of  $A^*$, the
dual of the quotient ring $A$ as a $K$-vector space, though still linearly independent.  Hence other forms must be considered in
order to describe $A^*$ and to get a non-singular matrix
generalizing $V_\TT(Z)$.

All along this section we will use the language of dual algebras to generalize
 Theorem \ref{subresultant2}
 for the multiple roots case
 (see  for instance in \cite{KK87,BCRS96} and the references therein).
 In Theorem~\ref{multtheorem} below we show that any basis of the dual $A^*$  gives
  rise to generalizations of   Theorem \ref{subresultant2}, as long as we assume that $\TT$ is a basis of $A$.
   This is the most general setting where a generalization of   Theorem \ref{subresultant2} will hold.
     However, this version of the Theorem, using general elements of the dual,
     does not give a formula for the subresultant in terms of the roots.

In order to obtain these expressions, we need to consider a specific basis of $A^*$
which contains the evaluation maps described above.
It turns out that one can define a basis for  $A^*$ in terms of  linear combinations of
 higher order derivative operators evaluated at roots of $f_1,\dots,f_n$.
 This is the content of the so called theory of ``inverse systems'' introduced by Macaulay in \cite{Mac16},
 and developed in a context closer  to our situation under the name of ``Gr\"obner duality " in \cite{Gr70,mamomo95,EM07}
  among others.

%First, we give a generalization of  Theorem \ref{subresultant2} to
%the multiple roots setting, or of Theorem~\ref{reph} to the
%multivariate case, by replacing the set
%$\{\mbox{ev}_{\bfxi_i}\}_{\,1\le i\le D}$ by any arbitrary basis of
%$A^*$.\\

\smallskip

The
following is a  multivariate analogue
of Definition~\ref{gW}:
\begin{definition}\label{D:MultHermite}
Let ${\bf\Lambda}:=\{\Lambda_1,\,\ldots,\,\Lambda_D\}$ be a basis of
$A^*$ as a $K$-vector space. Given any set
$E=\left\{\bfx^{\bfalpha_1}, \ldots, \bfx^{\bfalpha_u}\right\}$ of
$u$ monomials and given any polynomial $h\in K[\bfx],$  the {\em
generalized Vandermonde matrix} $V_{E}({\bf\Lambda})$ and the {\em
generalized Wronskian  matrix} $W_{h,E}({\bf\Lambda})$ corresponding
to $E$, ${\bf\Lambda}$ and $h$  are the following $u\times D$
matrices: {\small \begin{equation*}\label{E:Vandermondemult}
 V_{E}({\bf\Lambda}) = \begin{array}{|ccc|c}
\multicolumn{3}{c}{\scriptstyle{D}}\\
\cline{1-3}
\Lambda_1(\bfx^{\bfalpha_1}) &\cdots&\Lambda_D(\bfx^{\bfalpha_1})& \\
\vdots & & \vdots&  \scriptstyle u\\
\Lambda_1(\bfx^{\bfalpha_u}) &\cdots& \Lambda_D(\bfx^{\bfalpha_u})& \\
 \cline{1-3}
 \multicolumn{2}{c}{}
\end{array}  \
,\ W_{h,E}({\bf\Lambda})=\begin{array}{|ccc|c}
\multicolumn{3}{c}{\scriptstyle{D}}\\
\cline{1-3}
\Lambda_1(\bfx^{\bfalpha_1}h )&\cdots&\Lambda_D(\bfx^{\bfalpha_1}h)& \\
\vdots & & \vdots&  \scriptstyle u\\
\Lambda_1(\bfx^{\bfalpha_u}h )&\cdots& \Lambda_D(\bfx^{\bfalpha_u}h)& \\
 \cline{1-3}
 \multicolumn{2}{c}{}
\end{array}.
\end{equation*}}
\end{definition}

We modify  the definition of the matrix $\OO_\SS(Z)$ in (\ref{OS})
as follows:
\begin{definition}%\label{Os}
Let $\SS=\{\bfx^{\bfgamma_1},\ldots,\bfx^{\bfgamma_k}\}\subset
K[\bfx]_t$ be of cardinality  $k=\HH_{D_1\dots D_{n+1}}(t)$,
$\TT^*:=\cup_{j=t+1}^{\rho}\TT_j$ as in (\ref{prof})  and ${\mathcal
R}$ as in (\ref{R}).  Then {\small \begin{equation*}%\label{OS2}
\OO_\SS({\bf\Lambda}):=\begin{array}{|ccc|c}
\multicolumn{1}{c}{}&\multicolumn{1}{c}{\scriptstyle D}&
\multicolumn{1}{c}{}&
\\
\cline{1-3}
 & V_\SS({\bf\Lambda})& & \scriptstyle{k}\\%[-2mm]
\cline{1-3}
 & V_{\TT^*}({\bf\Lambda})& & \scriptstyle{s}\\%[-2mm
\cline{1-3}
 & W_{f_{n+1},\RR}({\bf\Lambda})& & \scriptstyle{r}\\%[-2mm]
 \cline{1-3} \multicolumn{2}{c}{}
\end{array}\ \in \ {K}^{D\times D}.
\end{equation*}}
\end{definition}

Note that by our assumption on $\TT$ being a basis of $A$ and
${\bf\Lambda}$ being a basis of $A^*$, we have $\det
\big(V_{\TT}({\bf\Lambda})\big)\neq 0$. The following is the
extension of Theorem~\ref{subresultant2} to the multiple roots case.

\begin{theorem}\label{multtheorem} Let $(f_1,\ldots.f_{n+1})\subset K[\bfx]$
and
 $ \TT:=\cup_{j\geq0}\TT_j$  specified in (\ref{prof}) satisfying our assumptions, and $\bf\Lambda$ be
 an arbitrary basis of  $A^*$. For any $t\in\Z_{\geq0}$ and  for any
$\SS=\{\bfx^{\bfgamma_1},\ldots,\bfx^{\bfgamma_k}\}\subset
K[\bfx]_t$ of cardinality  $k=\HH_{D_1\dots D_{n+1}}(t)$, we have
\begin{eqnarray*}%\label{multexpression}
\Delta_\SS(f_1,\dots,f_{n+1})=\pm \left(\prod_{j=t-D_{n+1}+1}^t
\widetilde \Delta_{\TT_{j}}\right)
\frac{\det\big(\OO_{\SS}({\bf\Lambda})\big)}{\det \big(
V_{\TT}({\bf\Lambda})\big)}.
\end{eqnarray*}
\end{theorem}
\smallskip

\begin{proof}[Proof of Theorem \ref{multtheorem}]
The proof is similar to the  proof of Theorem~$3.2$ in
\cite{DKS2006}, to which  we refer  for   notations and details.
Extra care must be taken however, as we are not anymore considering
the polynomials $f_1,\dots, f_n$ to have simple common roots.
\\
Using the exact same argument as in the proof of Theorem~$3.2$ in
\cite{DKS2006} we can prove that
$$\pm \,\EE(t) \,\Delta_\SS(f_1,\dots,f_{n+1}) \,\det\big(V_\TT({\bf\Lambda})\big)=\pm \det(M')\det\big(
\OO_\SS({\bf\Lambda})\big), $$ where
\begin{equation}\label{M'}M':=\begin{array}{|c|}\cline{1-1}M'_{f_1}\\ \vdots \\
\;\;M'_{f_n}\;\;\\ \cline{1-1}\end{array},
\end{equation}
is the submatrix of \eqref{sylv}
%\begin{equation}\label{sylv}
%\left[\begin{array}{c}\;\;M_{f_1}\;\;\\[-2mm]\vdots \\[-2mm]\;\;M_{f_{n}}\;\;\end{array}\right]
%\end{equation}
obtained by removing the columns corresponding the monomials in $\TT$. In \cite{DKS2006} we
also showed that
$$\det(M')=\pm \EE(t) \,\bigg(\prod_{j=t-D_{n+1}+1}^t
\widetilde \Delta_{\TT_{j}}\bigg),$$ so   the claim is proved when
$\EE(t)\neq0$.
\\
If $\EE(t)=0$, we consider a perturbation ``\`a la Canny'' as in
\cite{Can1990}, i.e. we replace $f_i$ by $f_{i,\lambda}:=f_i+
\lambda \,x_i^{D_i}\in K(\lambda)[\bfx]$, where $\lambda$ is a new
parameter,  for $ 1\le i\le n$. It is easy to see that this
perturbed system has no roots at infinity over the algebraic closure
$\overline{K(\lambda)}$ of $K(\lambda)$, since the leading term in $\lambda$ of the
resultant of its homogeneous components of degrees $D_1,\dots,D_n$
does not vanish, and   hence the dimension of the quotient ring
$A_\lambda:=K(\lambda)[\bfx]/(f_{1,\lambda},\ldots,f_{n,\lambda})$ as a  $K(\lambda)$-vector space
is also equal to $D$.
\\
It can also be shown (see \cite{Can1990}) that $\EE_\lambda(t)\ne
0$, where $\EE_\lambda(t)$ denotes the extraneous factor in
Macaulay's formulation applied to the polynomials
$f_{i,\lambda},\,1\le i\le n$. Indeed, if $E_t$ is the matrix whose
determinant gives $\EE(t)$ with rows and columns ordered properly,
it is easy to see that the perturbed matrix is equal to
$E_t+\lambda\,I$, where $I$ is the identity matrix.
\\
Therefore, the statement holds for this perturbed family:
\begin{equation}\label{mress}
\Delta_\SS(f_{1,\lambda},\dots,f_{n,\lambda},f_{n+1})=\pm
\left(\prod_{j=t-D_{n+1}+1}^t
\widetilde\Delta_{\TT_j,\lambda}\right)
\frac{\det\big(\OO_{\SS}({\bf\Gamma_\lambda})\big)}{\det
\big(V_{\TT}({\bf\Gamma_\lambda})\big)}
\end{equation}
for \textit{any} basis ${\bf\Gamma_\lambda}$ of $A_\lambda^*$ (here,
$\widetilde \Delta_{\TT_j,\lambda}=
\Delta_{\TT_j}^{(j)}(\widetilde{f}_{1,\lambda},\ldots,\widetilde{f}_{n,\lambda})$).\\
The subresultants appearing in \eqref{mress} are polynomials in
$\lambda $, that, when evaluated in $\lambda=0$, satisfy:
$$
\left.\Delta_\SS(f_{1,\lambda},\dots,f_{n,\lambda},f_{n+1})\right|_{\lambda=0}=\Delta_\SS(f_{1},\dots,f_{n},f_{n+1}),
\
\left.\widetilde\Delta_{\TT_j,\lambda}\right|_{\lambda=0}=\widetilde\Delta_{\TT_j},
\ \forall j.
$$ So, in order to prove the claim, it is enough to show  that there exists a
basis
 of $A_\lambda^*$ which ``specializes'' to ${\bf\Lambda}$
when setting $\lambda=0$, i.e. to find a basis
${\bf\Lambda}_\lambda$ of $A_\lambda^*$ such that
\begin{equation}\label{final}
\left.\frac{\det\big(\OO_{\SS}({\bfLambda_\lambda})\big)}{\det
\big(V_{\TT}({\bfLambda_\lambda})\big)}\right|_{\lambda=0}=\frac{\det\big(\OO_{\SS}({\bfLambda})\big)}{\det
\big(V_{\TT}({\bfLambda})\big)},
\end{equation}
and then to apply  Identity \eqref{mress} to ${\bfLambda_\lambda}$
and to specialize it at $\lambda=0$.

\smallskip
\noindent We now construct the basis $\bf\Lambda_\lambda$: The
monomial basis
$\TT=\{\bfx^{\bfalpha_1},\ldots,\bfx^{\bfalpha_{D}}\}$  of $A$
is also a monomial basis of  $A_\lambda$, since clearly linearly
independent, and therefore it defines the dual bases
$\{\bfy_{\bfalpha_1},\ldots,\bfy_{\bfalpha_{D}}\}\subset A^*$ and
$\{\bfy_{\bfalpha_1,\lambda},\ldots,\bfy_{\bfalpha_{D}
,\lambda}\}\subset A^*_\lambda$, satisfying for   $1\le j, k\le D$,
$$
\bfy_{\bfalpha_k}(\bfx^{\bfalpha_j})=\bfy_{\bfalpha_{k},\lambda}(\bfx^{\bfalpha_{j}})=1
\
 \mbox{ if } \  k=j \ \mbox{ and   } \  0 \ \mbox{ otherwise }.
$$
We write ${\Lambda}_i=\sum_{k=1}^Dc_{ik}{\bfy_{\bfalpha_k}}$ for
$1\le i\le D$, where $c_{ik}\in K$, and then set
${\bf\Lambda}_\lambda:=\{{\Lambda}_{1,\lambda},\ldots,{\Lambda}_{D,\lambda}\}$,
with
$\Lambda_{i,\lambda}:=\sum_{k=1}^Dc_{ik}{\bfy_{\bfalpha_k,\lambda}},
\ 1\le i\le  D$. Note that
\begin{eqnarray}
&\Lambda_{i,\lambda}(\bfx^{\bfalpha_j})=\Lambda_{i}(\bfx^{\bfalpha_j})=c_{ij}
\ \mbox{ for } \ 1\le i,j\le D, \label{mress2prime} \\  &\det
\big(V_{\TT}({\bf\Lambda_\lambda})\big)=\det
\big(V_{\TT}({\bf\Lambda})\big)=\det\big(c_{ij}\big)_{1\leq i,j\leq
D}\in K\setminus\{0\}, \label{mress2}
\end{eqnarray}
as the matrix $(c_{ij})_{1\leq i,j\leq D}$ is invertible. This
implies that ${\bf\Lambda_\lambda}$ is a basis of $A_\lambda^*$.\\
We claim now that, for every $\bfalpha\in\N^n,$ there exist
polynomials $p_\bfalpha$ and $ \,A_{j,\bfalpha} $, $1\le j\le D$, in
$ K[\lambda]$    such that
\begin{equation}\label{mress4}
p_\bfalpha(\lambda)\,{\bfx}^\bfalpha=\sum_{j=1}^DA_{j,\bfalpha}(\lambda){\bfx}^{\bfalpha_j}\
\mbox{ in } \ A_\lambda, \quad \mbox{and} \quad p_\bfalpha(0)\neq0.
\end{equation}
%We claim now that, for every $\bfalpha\in\N^n,$ there exist
%polynomials $p_\bfalpha ,\,A_{j,\bfalpha} $ and $B_{i,\bfalpha} \in
%K[\lambda],$   $1\leq j\leq D,\,1\leq i\leq n$, such that
%\begin{equation}\label{mress4}
%p_\bfalpha(\lambda)\,{\bfx}^\bfalpha=\sum_{j=1}^DA_{j,\bfalpha}(\lambda){\bfx}^{\bfalpha_j}+
%\sum_{i=1}^nB_{i,\bfalpha}(\lambda)\,f_{i,\lambda} \quad \mbox{and}
%\quad p_\bfalpha(0)\neq0.
%\end{equation}
For  this, it suffices to express the monomial
 ${\bfx}^\bfalpha$ in terms of the basis ${\mathcal T}$ of
 $A_\lambda$ and take
$p_{\bfalpha}(\lambda)$ as a common denominator when lifting the
expression
 to $K[\lambda][{\bfx}]$,  satisfying  the condition
$\gcd\big(p_\bfalpha ,\,A_{j,\bfalpha} , 1\le  j\le D)=1$. It is
clear that $p_\bfalpha(0)\ne 0$ because $\TT$ is also a basis of
$A$, and by assumption $0$ is not a common root of $p_\bfalpha$ and
$A_{j,\bfalpha}$, $1\le j\le D$.
% For this, it suffices to express the
%monomial
% ${\bfx}^\bfalpha$ in terms of the basis ${\mathcal T}$ modulo the ideal
% $\big( f_{1,\lambda},\ldots, f_{n,\lambda}\big)$ in
% $K(\lambda)[{\bfx}]$, and take
%$p_{\bfalpha}(\lambda)$ as a common denominator when lifting the
%expression
% to $K[\lambda][{\bfx}]$, moreover satisfying  the condition
%$\gcd\big(p_\bfalpha ,\,A_{j,\bfalpha} ,\,B_{i,\bfalpha}; \, \forall
%\,j,i \big)=1$. We still need to prove that $p_\alpha(0)\ne 0$.
%Suppose that $p_\alpha(0)=0$. Then, by setting $\lambda=0$ in
%Identity \eqref{mress4}, we get
%$$0=\sum_{j=1}^DA_{j,\bfalpha}(0){\bfx}^{\bfalpha_j}+\sum_{i=1}^nB_{i,\bfalpha}(0)\,f_i,
%$$
%which implies $A_{j,\bfalpha}(0)=0$ for $1\le j\le D$ since $\TT$ is
%a basis of $A$.  Then $0=\sum_{i=1}^nB_{i,\bfalpha}(0)\,f_i$ implies
%that $B_{i,\bfalpha}(0)=0$  for $1\le i\le n$ since $f_1,\dots,f_n$
%cannot be linearly dependent as they define a zero-dimensional
%ideal. This would imply that
%$p_\bfalpha,\,A_{j,\bfalpha}\,B_{i,\bfalpha}$ share a root, which
%contradicts the assumption
%$\gcd\big(p_\bfalpha,\,A_{j,\bfalpha},\,B_{i,\bfalpha}\big)=1$.
\\
Applying $\Lambda_{i,\lambda}$ to  Identity~\eqref{mress4} and
$\Lambda_i$ to Identity~\eqref{mress4} specialized at $\lambda=0$,
we then get by \eqref{mress2prime} for $ 1\le i\le D$:
$$
p_\bfalpha(\lambda)\Lambda_{i,\lambda}({\bfx}^\bfalpha)=\sum_{j=1}^D
c_{ij} A_{j,\bfalpha}(\lambda) \ \mbox{ and } \
p_\bfalpha(0)\Lambda_{i}({\bfx}^\bfalpha)=\sum_{j=1}^D c_{ij}
A_{j,\bfalpha}(0).
$$
This implies that the entries of the matrix
$\OO_{\SS}({\bf\Lambda_\lambda})$ are the same $K$-linear
combinations of the quotients
$\frac{A_{j,\bfalpha}(\lambda)}{p_\bfalpha(\lambda)}$ than the
entries of the matrix $\OO_{\SS}({\bf\Lambda})$ in terms of
$\frac{A_{j,\bfalpha}(0)}{p_\bfalpha(0)}$. This  and Identity
\eqref{mress2}  implies \eqref{final}, which proves the statement.
\end{proof}

 As we mentioned before, for an arbitrary basis ${\bf\Lambda} $ of $A^*$ the expression in Theorem~\ref{multtheorem} may not provide a formula  in terms of the roots of $f_1, \ldots, f_n$.
 In order to obtain one, we recall here the notion of Gr\"obner duality from
 \cite{mamomo95}.\\
 For $\bfalpha=(\alpha_1, \ldots, \alpha_n)\in \N^n$ define the differential operator
 $$
 \partial_{\bfalpha} :=
  \frac{1}{\alpha_1!\cdots  \alpha_n!}\,\frac{\partial ^{|\bfalpha|}}
  {\partial x_1^{\alpha_1}\cdots \partial x_n^{\alpha_n}}
 $$
  and consider the ring $K[[\partial]]:=\{\sum_{\bfalpha\in \N^n} a_\bfalpha\partial_\bfalpha : a_\bfalpha\in
  K\}$.

  \smallskip \noindent
For $1\le i\le  n$ define the $K$-linear map
$$\sigma_i:
K[[\partial]] \to K[[\partial]]; \
\sigma_i\left(\partial_{\bfalpha} \right)=\begin{cases} \partial_{(\alpha_1, \ldots, \alpha_i-1,\ldots, \alpha_n)} & if \; \alpha_i>0,\\
0 &\text{otherwise},
\end{cases}
$$
and for $\bfbeta=(\beta_1, \ldots, \beta_n)\in \N^n$ define
$\sigma_\bfbeta= \sigma_1^{\beta_1}\circ
\cdots\circ\sigma_n^{\beta_n}$.

\smallskip\noindent
A K-vector space $V\subset K[[\partial]]$ is {\em closed} if
$\dim_K(V)$ is finite  and for all $\bfbeta\in \N^n$ and ${\bf D}
\in V$ we have  $\sigma_\bfbeta({\bf D})\in V$.  Note that
$K[[\partial]]$ and its closed subspaces have a natural
$K[{\bfx}]$-module structure given by ${\bfx}^\bfbeta {\bf
D}(f):={\bf D}({\bfx}^\bfbeta f)= \sigma_\bfbeta({\bf D})(f)$.

\smallskip \noindent
Let  $\bfxi \in {K}^n$. For a closed subspace $V\subset
K[[\partial]]$ define
$$
\nabla_\bfxi(V):=\{f\in  K[\bfx] \;:\; {\bf D}(f)(\bfxi)=0
,\;\;\forall \,{\bf D}\in V\}\ \subset \ K[\bfx].
$$
Let ${\bf m}_\bfxi\subset  K[{\bfx}]$  be the maximal ideal defining
$\bfxi$. For an ideal $J\subset  {\bf m}_\xi$ define
$$
\Delta_\bfxi(J):=\{{\bf D}\in  K[[\partial]]\;:\; {\bf
D}(f)(\bfxi)=0, \;\;\forall \,f\in J\} \ \subset \ K[[\partial]].
$$
Then the following theorem gives the so called Gr\"obner duality:

\begin{theorem}[\cite{Gr70,mamomo95}] Fix $\bfxi \in {K}^n$.
The correspondences between  closed subspaces $V\subset
K[[\partial]]$ and ${\bf m}_\xi$-primary ideals $Q$, $V\mapsto
\nabla_\bfxi(V)$ and  $ Q \mapsto \Delta_\bfxi(Q)  $ are 1-1 and
satisfy $V=\Delta_\bfxi(\nabla_\bfxi(V))$ and
$Q=(\nabla_\bfxi(\Delta_\bfxi(Q)))$. Moreover,
$$
\dim_{K}(\Delta_\bfxi(Q))={\rm mult}(Q)\quad \mbox{and}\quad {\rm
mult}(\nabla_\bfxi(V))=\dim_{K}(V).
$$
\end{theorem}

 \noindent
We  set  $Z:=\{{\bfxi}_1,\dots, {\bfxi}_{D}\}$ for  the set of all
common roots of $f_1,\dots, f_n$ in $ K^n$, with multiplicities, and
${\bf m}_{\bfxi_i}\subset   K[{\bfx}]$ for  the maximal ideal
corresponding to $\bfxi_i$ for $1\le i\le m$.

   \begin{example}
 (\cite[Exemple $7.37$]{EM07}) \\Let
$f_1=2x_1x_2^2+ 5x_1^4$, $f_2=2x_1^2x_2+5x_2^4\ \in \C[x,y]$.\\ Then
$Z=\{ {\mathbf{0}},\bfxi_1,\bfxi_2,\bfxi_3,\bfxi_4,\bfxi_5\}$ where
$ { \mathbf{0}} =(0,0)$ has  multiplicity eleven and
$\bfxi_i=(\frac{-2}{5\xi^{2i}},\frac{-2}{5\xi^{3i}})$ where $\xi$ is
a primitive 5-th root of unity, are all simple,  $1\le i\le 5$. \\
Denote by $Q_{\mathbf{0}}$, $Q_{\bfxi_i}$,  $1\le i\le 5$, the
primary ideals corresponding to the roots, then
$\Delta_{\bfxi_i}(Q_{\bfxi_i})=\langle 1\rangle$ for $1\le i\le 5$
and if  $\{\bfe_1,\bfe_2\}$ is the canonical basis of $\Z^2$,
\begin{align*} \Delta_{\mathbf{0}}(Q_{\mathbf{0}})=&\langle
1 ,\partial_{\bfe_1} ,\,\partial_{\bfe_2} ,\,\partial_{2\bfe_1} ,\,
\partial_{\bfe_1+\bfe_2} ,\,\partial_{2\bfe_2} ,\,\partial_{3\bfe_1} ,
\,\partial_{3\bfe_2} ,
\\ &
 (4\partial_{4\bfe_1}-5\partial_{\bfe_1+2\bfe_2}) ,\,
 (4\partial_{4\bfe_2}-5\partial_{2\bfe_1+\bfe_2}) ,
(3\partial_{2\bfe_1+3\bfe_2}-\partial_{5\bfe_1}-\partial_{5\bfe_2})
\rangle.
\end{align*}
\end{example}

Using Gr\"obner duality, we are now able to give an expression for
the subresultant in terms of the roots of $f_1, \ldots, f_n$.
 For ${\bf D}\in
  K[[\partial]]$and $\bfxi\in K^n$, we denote by
   ${\bf D}|_\bfxi$ the element of $A^*$ defined as  ${\bf D}|_\bfxi(f)={\bf D}(f)(\bfxi)$.
   In particular, under this notation, $1|_\bfxi =
   \mbox{ev}_{\bfxi}$.

\begin{corollary}
Using our previous assumptions, let  $I=(f_1, \ldots ,f_n)$ and
$$
I =Q_1\cap \cdots \cap Q_m
$$
be the primary decomposition of $ I $, where $Q_i$ is a ${\bf
m}_{\bfxi_i}$-primary ideal with $d_i:={\rm mult}(Q_i)$. For $1\le
i\le  m$ let $V_i:=\Delta_{\bfxi_i}(Q_i)\subset  K[[\partial]]$ be
the corresponding closed subspace, and fix a basis $ \{{\bf
D}_{i,1}, \ldots, {\bf D}_{i,d_i}\}$ for $V_i$ such that ${\bf
D}_{i,1}=1$.  Then
 $${\bf \Lambda}:=\{{\bf
D}_{1,1}|_{\bfxi_1},\ldots, {\bf D}_{1,d_1}|_{\bfxi_1}, \ldots, {\bf
D}_{m,1}|_{\bfxi_m},\ldots,{\bf D}_{m,d_m}|_{\bfxi_m}\}$$ is a basis
of $A^*$ over ${K}$.
\end{corollary}
Note that the  above choice for the dual basis ${\bf \Lambda}$  contains the evaluation maps for the roots of $I$, and using  this ${\bf \Lambda}$ in Theorem~\ref{multtheorem}  gives
an expression for the subresultant in terms of the roots of $I$.

\begin{example}
This is a very simple example containing an expression for a
subresultant in terms of the roots. \\Let $f_1:=x_1x_2,
\;f_2:=x_1^2+(x_2-1)^2-1,\; f_3:= c_0+c_1x_1+c_2x_2,$ with
$c_0,\,c_1,\,c_2\in \C$. Then $Z=\{(0,0),(0,2)\}$ where $(0,0)$ has
multiplicity 3 and $(0,2)$ is  simple.
 By computing explicitly, we check that $\TT:=\{1, x_1, x_2, x_2^2\}$ is a basis of
 $A=\C[x_1,x_2]/( f_1, f_2)$  and that
$${\bf \Lambda}:=\left\{1|_{(0,0)}, \partial_{\bfe_1}|_{(0,0)} , (\partial_{\bfe_2}+2\partial_{2\bfe_1}) |_{(0,0)}, 1|_{(0,2)} \right\}
$$
is a basis of $A^*$. We will use these bases to express the degree
$t=\rho=2$ subresultant  $\Delta_{x_1^2}(f_1, f_2, f_3)$ with
$\SS=\{x_1^2\}$ in terms of the roots of $f_1, f_2$. First,
$\Delta_{x_1^2}(f_1, f_2, f_3)$ is equal to the following $6\times
6$ determinant (since here the extraneous factor is 1):
$$
\Delta_{x_1^2}(f_1, f_2, f_3)=\det M_\SS=\det \begin{array}{|cccccc|}
\cline{1-6}
0&0&0&1&0&0\\
0&0&-2&1&0&1\\
0&0&0&0&1&0\\
c_0&c_1&c_2&0&0&0\\
0&c_0&0&c_1&c_2&0\\
0&0&c_0&0&c_1&c_2\\
\cline{1-6}
\end{array}= c_0^3+2c_0^2c_2.
$$
On the other hand, Theorem~\ref{multtheorem} gives the following expression:
$$
\left(\prod_{j=2}^2
\widetilde\Delta_{\TT_j}\right)
\frac{\det\OO_{\SS}({\bf\Lambda})}{\det V_{\TT}({\bf\Lambda})}=\frac{\det\begin{array}{|cccc|}
\cline{1-4}
0&0&2&0\\
c_0&c_1&c_2&c_0+2c_2\\
0&c_0&2c_1&0\\
0&0&c_0&2c_0+4c_2\\
\cline{1-4}
\end{array}}{\det
\begin{array}{|cccc|}
\cline{1-4}
1&0&0&1\\
0&1&0&0\\
0&0&1&2\\
0&0&0&4\\
\cline{1-4}
\end{array}}=\frac{4(c_0^3+2c_0^2c_2)}{4},
$$
using that $\TT_2=\{x_2^2\} $ is the degree 2 part of $\TT$, and
$$
\widetilde \Delta_{\TT_{2}}(\widetilde{f}_1,\widetilde{f}_2)=\det
\begin{array}{|ccc|}
\cline{1-3}
0&0&1\\
1&0&1\\
0&1&0\\
\cline{1-3}
\end{array}=1.
$$
\end{example}

\end{document}